\documentclass[ejs]{imsart}

\RequirePackage[OT1]{fontenc}
\RequirePackage{amsthm,amsmath}
\RequirePackage[numbers]{natbib}
\RequirePackage[colorlinks,citecolor=blue,urlcolor=blue]{hyperref}
\usepackage{graphicx}
\usepackage{amsfonts}
\usepackage{amssymb}
\usepackage{latexsym}
\usepackage{amsthm}
\usepackage{pifont}
\usepackage{graphicx}
\usepackage{epsfig}
\usepackage{hyperref}
\usepackage{psfrag}
\usepackage{url}
\usepackage{xcolor}

\pubyear{2020}
\volume{0}
\issue{0}
\firstpage{1}
\lastpage{8}
\arxiv{2006.08010}

\startlocaldefs
\numberwithin{equation}{section}
\theoremstyle{plain}
\newtheorem{theorem}{Theorem}[section]
\newtheorem{definition}[theorem]{Definition}
\newtheorem{example}[theorem]{Example}
\newtheorem{hyp}{Assumption}
\newtheorem{proposition}[theorem]{Proposition}

\newtheorem{lemma}[theorem]{Lemma}
\newtheorem{remark}[theorem]{Remark}

\newcommand{\E}{{\mathbb E}}
\newcommand{\N}{{\mathbb N}}

\renewcommand{\P}{{\mathbb P}}

\newcommand{\R}{{\mathbb R}}
\newcommand{\ind}{{\bf 1}}
\newcommand{\Card}{{\rm Card}}

\newcommand{\sub}{{\rm sub}}
\newcommand{\class}{{\rm class}}
\newcommand{\Lag}{\mathcal{L}{\rm ag}}

\newcommand{\x}{{\mathbf x}}

\newcommand{\lbrac}{[\![}
\newcommand{\rbrac}{]\!]}

\def\tvc#1{\color{black}#1 \color{black}}
\def\thuy#1{\color{black}#1 \color{black}}

\endlocaldefs

\begin{document}
	
	\begin{frontmatter}
		\title{Estimation of dense stochastic block models visited by random walks}
		\runtitle{Estimation of dense SBM with RW}
		\thankstext{T1}{The authors thank the two anonymous Referees whose comments contributed much to improve the paper. Especially, the remark of one of the Referee directly pointed to the solution of Section \ref{sec:Referee}. They are also indebted to Jean-St\'ephane Dhersin, Sophie Donnet, St\'ephane Robin, Adrian R\"ollin and Timoth\'ee Tabouy for discussions. This work was supported by the GdR GeoSto 3477, by the ANR Econet (ANR-18-CE02-0010) and by the Chair ``Modélisation Mathématique et Biodiversité" of Veolia Environnement-Ecole Polytechnique-Museum National d'Histoire Naturelle-Fondation X.}
		
		\begin{aug}
			\author{\fnms{Viet Chi} \snm{Tran}\ead[label=e1]{chi.tran@univ-eiffel.fr}}
			\and
			\author{\fnms{Thi Phuong Thuy} \snm{Vo}\ead[label=e2]{phuongthuywz@gmail.com}}
			\address{LAMA, Univ Gustave Eiffel, Univ Paris Est Creteil, CNRS, F-77454 Marne-la-Vall\'ee, France\\
				\printead{e1,e2}}
			\runauthor{Tran Viet Chi and Vo Thi Phuong Thuy}
			\affiliation{Universit\'e Gustave Eiffel}
		\end{aug}
		
		\begin{abstract}
			We are interested in recovering information on a stochastic block model from the subgraph discovered by an exploring random walk. Stochastic block models correspond to populations structured into a finite number of types, where two individuals are connected by an edge independently from the other pairs and with a probability depending on their types. We consider here the dense case where the random network can be approximated by a graphon. This problem is motivated from the study of chain-referral surveys where each interviewee provides information on her/his contacts in the social network. First, we write the likelihood of the subgraph discovered by the random walk: biases are appearing since hubs and majority types are more likely to be sampled. Even for the case where the types are observed, the maximum likelihood estimator is not explicit any more. When the types of the vertices is unobserved, we use an SAEM algorithm to maximize the likelihood. Second, we propose a different estimation strategy using new results by Athreya and R\"ollin. It consists in de-biasing the maximum likelihood estimator proposed in Daudin et al. and that ignores the biases.\end{abstract}
		
		
		\tableofcontents
	\end{frontmatter}

	\section{Introduction}
	
	A way to infer a random structure such as the graph of a social network and discover its properties is to explore it with random walks (e.g. \cite{riordan}). This mathematical idea can be put into practice to reveal hidden populations such as drug users by using referral chain sampling where each new person provides information on her/his contacts: see for example the snowball sampling \cite{goodman} or the `respondent-driven sampling' (RDS) introduced by Heckathorn \cite{heckathorn} (see also the PhD thesis of the second author \cite{vo-these}). These methods were first used to estimate the size of the hidden population or to infer population means, under the assumption that subjects' network degree determines their probability of being sampled, see Volz and Heckathorn \cite{volzhecathorn} (see also \cite{lirohe}). Because the inclusion probability of a subject is complicated to compute, due to the dependencies associated with the graph and the fact that the sampling should be in practice without replacement, an important numerical literature on the subject has followed (see e.g. \cite{gilehandcock2010,gilejohnstonsalganik,mouwverdery}). Gile \cite{gile2011} proposed an improved estimator for population means taking into account the without replacement sampling, and Rohe established critical threshold for the design effects \cite{rohe}. Because of privacy restrictions, the social-network information is usually only a tree, as each interviewee has been `invited' into the survey by a previously interviewed subject. Crawford, Wu and Heimer \cite{crawfordwuheimer} use a Bayesian approach to integrate over the missing edge between recruited individuals.\\
	It appears that the information gathered in chain-referral surveys can also be used in estimating the social network itself or at least properties associated with its topology. Recent surveys allow to gather connectivity information for recruited members: see for example the Rolls et al. \cite{rolls2013} and Jauffret-Roustide et al. \cite{jauffretroustide-enquete}. Interviewees are asked for a description of their contacts, and for a first name or a nickname. This information allows to reconstruct partially the social network and obtain a subgraph that is not a tree. It is then natural to wonder how much information on the total graph can be recovered from the observation of the subgraph obtained by the chain-referral sampling. Of course, biases have been emphasized as individuals of high degrees (hubs) are sampled with higher probability and `common profiles' are much more likely to be discovered (e.g. \cite{khabbazianhanlonrussekrohe}). This motivates the present paper. To fix the framework of study, we consider a particular class of random graphs, namely the Stochastic Block Models (SBM) that are popular models for social networks (see \cite{hollandlaskeyleinhardt} and the review \cite{abbe}). For this parametric model, inferring the distribution of the random graph boils down to a finite dimensional parameter estimation. Also, for simplification, we consider here a model of random walk on the continuous version of the SBM graph, namely the SBM graphon that is introduced in the next paragraph. Two estimations strategies are considered in this paper. First, we establish the likelihood of a random walk exploring this structure, and which accounts for the sampling biases. Two cases are classically considered, depending on whether the types of the visited nodes are observed or not. Even in the case of a complete observation, the maximum likelihood estimator has no explicit form. When the types of the vertices are unobserved, we adapt the Stochastic Approximation Expectation-Maximization algorithm (SAEM) as introduced in \cite{celeux:inria,kuhnlavielle}. Second, we propose a new estimation using new theoretical probabilistic results by Athreya and Roellin \cite{athreyaroellin1} who compute an exact formula for the bias. We provide a consistent estimator in the case of complete observations and a de-biasing strategy for the usual maximum likelihood estimator of Daudin et al. \cite{daudinpicardrobin} in the case where the types of the explored nodes are unknown.\\

	We consider as a toy model a Stochastic Block Model graphon with $Q$ classes. Graphons, considered here as symmetric integrable functions from $[0,1]^2$ to $\R$, can be seen as limit of dense graphs (see e.g. \cite{lovaszbook}). Recall that SBM graphs are a generalization of Erd\"os-Rényi graphs, where each node $i$ is characterized by a type, $Z_i\in \{1,\dots, Q\}$, with $Q$ the number of different possible values. The random variable (r.v.) $Z_i$ are assumed independent and identically distributed (i.i.d.) with $\P(Z_i=q)=\alpha_q>0$. Each pair of nodes $\{i,j\}$ is connected independently with a probability $\pi_{Z_i,Z_j} \in (0,1)$ that depends only on the types. \tvc{Because the graph is non oriented, the matrix with entries $\pi_{qr}$ is symmetric ($\pi_{qr}=\pi_{rq}$). Thus, for a given $Q$, the distributions of SBM graphs are parameterized by the vector
		\[\theta=(\alpha_q,\pi_{qr}, ; q,r\in \{1,\cdots Q\}).\]} When the number of vertices of the graph tends to infinity, it is known that the dense graph converges to a limiting continuous object called graphon, see e.g. \cite{borgschayeslovaszsosvesztergombi08, borgschayeslovaszsosvesztergombi12, lovaszbook}. Let us recall the definition of the SBM graphon.
	
	For the sequel, we introduce the partition of $[0,1]$ defined by
	\begin{equation}\label{def:Iq}
	I_q=\big[ A_{q-1},A_q),\qquad q\in \{1,\dots Q\}
	\end{equation}
	where for $q\in \{1,\dots Q\}$, $A_q= \sum_{k=1}^{q} \alpha_k$, with $A_0=0$ by convention.
	The SBM graphon $\kappa_{\theta}$, associated with the parameter $\theta=(\alpha_q,\pi_{qr}, ; q,r\in \{1,\cdots Q\})$, is the function from $[0,1]^2$ to $[0,1]$ defined as follows:	
	\begin{equation}\label{grapho:SBM}
	\kappa_{\theta}(x,y)=\sum_{q=1}^Q\sum_{r=1}^Q \pi_{qr}\ \ind_{I_q}(x) \ind_{I_r}(y).
	\end{equation}Heuristically, we can see $[0,1]$ as a continuum of vertices, and the graphon is the limit of the \tvc{expectation of the} adjacency matrix of the graph in the sense that $\kappa_{\theta}(x,y)$ measures the probability of connection between $x$ and $y$. \\

	We consider a random walk on the graphon $\kappa_{\theta}$, i.e. the process $X=(X_m)_{m\geq 1}$ with values in $[0,1]$ and transition kernel:
	\begin{equation}
	K_\theta(x,dy)=\frac{\kappa_{\theta}(x,y)dy}{\int_0^1 \kappa_{\theta}(x,v) dv}=\frac{\sum_{q=1}^Q \big(\sum_{r=1}^Q \pi_{qr}\ \ind_{I_r}(y) \Big) \ind_{I_q}(x) \ dy}{\sum_{q=1}^Q \Big(\sum_{r=1}^Q \pi_{qr} \alpha_r \Big) \ind_{I_q}(x)}.
	\end{equation}This random walk is the analogous of the classical random walk on a graph that jumps from a vertex to one of its neighbouring vertices chosen uniformly at random. \tvc{One simplification brought by studying the random walk on the graphon lies in the facts that (i) nodes can be visited only once and the random walk does not return to previously explored nodes, (ii) the Markov chain can not get stuck as would an avoiding random walk on a discrete graph. }\\
	From the exploration of this random walk, we can construct a subgraph of the `nodes' visited. Assume that we observe $n$ steps of the random walk, i.e. $X^{(n)} = (X_1,\dots, X_n)$. The associated path (up to its $n$th step) is a subgraph (chain) $H_n=(V_n, E_n)$ with vertices $V_n=\{X_1,\dots X_n\}$ and edges  $E_n = \cup_{m=1}^{n-1} \{X_m,X_{m+1}\}$. This chain is completed by sampling independently edges between vertices that are not already connected with probability according to their types. \tvc{We denote by $(Y_{ij} ; i,j\in \{1,\dots n\})$ the adjacency matrix of the resulting graph, i.e. $Y_{ij}=1$ if and only if $i\sim_{G_n} j$. Because the graph is non-oriented, we have $Y_{ij}=Y_{ji}$. Moreover, notice that by construction, we always have $Y_{i,i+1}=1$ for $i\in \{1,\dots n-1\}$.} Following the notation of Athreya and R\"ollin \cite{athreyaroellin1}, we denote by $G_n:=G(X^{(n)}, \kappa_\theta, H_n)$ the random graph, which is completed from $H_n$ w.r.t. the graphon $\kappa_\theta$:
	\begin{definition}\label{def:Gn}
		The vertices of $G_n=G(X^{(n)}, \kappa_\theta, H_n)$ are the nodes $X^{(n)}$, and the edges are as follows. Let $i$ and $j$ be two vertices.
		\begin{itemize}
			\item If there is an edge between $i$ and $j$ in $H_n$, $i\sim_{H_n} j$ then there is also an edge between these nodes in $G_n$: $i\sim_{G_n} j$.
			\item If there is no edge between $i$ and $j$ in $H_n$, we connect $i$ and $j$ in $G_n$ with probability $\kappa_\theta(X_i,X_j)$.
		\end{itemize}
	\end{definition}
	This subgraph $G_n$ is the RDS graph. Notice that the random walk and the subgraph $G_n$ can be defined for general graphons and not only SBM graphons (see \cite{athreyaroellin1}). \\
	
	In the rest of the paper, we assume that this is the model generating our data and that the observation corresponds to a realization of $G_n$.  \tvc{The complete data consists in:
		\begin{itemize}
			\item the chain $X^{(n)}=(X_i)_{i\in \{1,\cdots n\}}$ in $[0,1]$,
			\item the types of the successive vertices visited $Z=(Z_i)_{i\in \{1,\cdots n\}}$
			\item the adjacency matrix of $G_n$: $Y=(Y_{ij})_{i,j\in \{1,\cdots n\}}$ where $Y_{ij}=\ind_{i\sim_{G_n} j}$.
		\end{itemize}We will consider both the cases where (i) all these elements are observed, and the case where only a partial information is available: (ii) the adjacency matrix $(Y_{ij})_{i,j\in \{1,\cdots n\}}$ and the positions $X_i$'s of the vertices are observed, but not the $Z_i$'s. Notice that in the latter case, some information on the types $Z_i$'s can still be recovered since the latter depend on the $X_i$'s. (iii) only the adjacency matrix $(Y_{ij})_{i,j\in \{1,\cdots n\}}$ is observed.}\\
	Our purpose is to estimate $\theta=(\alpha_q,\pi_{qr} ; q,r\in \{1,\cdots Q\})$ using the subgraph $G_n$. In the literature, the estimation of SBM graphs has been extensively studied, but often in a framework where the number of nodes is known. In particular, variational EM approaches have been used in many cases where types are unknown, see \cite{daudinpicardrobin,tabouybarbillonchiquet,mariadassoutabouy}. The estimation of SBM graphs, when the total population size is unknown and when we only have a subgraph obtained by a chain-referral method, is not studied to our knowledge. We develop in this paper two approaches that we compare in a final numerical section (Section \ref{sec:numresults}). \\
	
	\tvc{For the first approach, it is possible to write the likelihood of $G_n$. Here, because graph is explored through an RDS random walk, our likelihood differs from the likelihoods in these papers: it accounts both on the transitions of the random walk and on the connectivity of vertices given their types. We study in Section \ref{sec:likelihood} the maximum likelihood estimator (MLE) in our setting for both cases, when the nodes types are observed (Section \ref{sec:completeobs}) or not (Section \ref{sec:likelihood-SAEM}). Even when the observation is complete, the maximum likelihood estimator does not have an explicit form. When the types are unknown, we adapt to our likelihood the variational EM approach of \cite{daudinpicardrobin}.\\
		The second approach developed in Section \ref{sec:graphon} is inspired by the recent work of Athreya and R\"ollin \cite{athreyaroellin1}. These authors showed that when we observe the random walk sufficiently long ($n\rightarrow +\infty$), the sequence of graphs $(G(H_n,\kappa_\theta))_{n\geq 1}$ converges to a biased graphon of $\kappa_\theta$. Based on their probabilistic result, a natural estimator of the biased graphon turns out to be the MLE in the `classical' case studied by \cite{daudinpicardrobin}. Based on this estimator that is not consistent in our case, we propose a new consistent estimator of $\theta$. We first detail the estimation for the case of complete observations (Section \ref{sec:completeobs2}) and then extend the variation EM of the first approach to this case (Section \ref{sec:incomplete_graphon}). Another possibility without using the information on the $X_i$'s is developed in Section \ref{sec:Referee}.
	}

	\section{Probabilistic setting}\label{sec:probsetting}

	In this section, we give some important properties of the RDS Markov chain $X^{(n)}$, in particular on its long term behaviour. Then we explain the biases that appear when estimating the graphon $\kappa_{\theta}$ from the RDS subgraph $G_n$.
	
	\subsection{Exploration by a random walk}
	
	\begin{hyp}\label{hypotheses}
		In all the paper, we consider the graphon $\kappa_{\theta}$ of an SBM graph (see \eqref{grapho:SBM}) and we assume that $\kappa_{\theta}$ is \textit{connected}, i.e. that for all measurable subset $A\subset [0,1]$ such that its Lebesgue measure $|A|\in (0,1)$,
		\tvc{\begin{align}0< & \int_A\int_{A^c} \kappa_{\theta}(x,y)dx\ dy
			=  \sum_{q=1}^Q \sum_{r=1}^Q \pi_{qr} |I_q\cap A| \ |I_r\cap A^c|,\label{eq:1}
			\end{align}using \eqref{grapho:SBM}.}	\end{hyp}

	\tvc{Let us now introduce some notations:
		\begin{equation}\bar{\pi}_q=\sum_{r=1}^Q \pi_{qr}\alpha_r ,\qquad \bar{\pi}=\sum_{q=1}^Q \bar{\pi}_q\alpha_q=\sum_{q=1}^Q\sum_{r=1}^Q \pi_{qr}\alpha_q\alpha_r. \label{notation:pibar}\end{equation}
		The quantity $\bar{\pi}_q$ corresponds to the mean connectivity of a node of class $q$ and $\bar{\pi}$ corresponds to the mean connectivity of a node chosen uniformly in $[0,1]$.
	}
	
	\begin{proposition}Under Assumptions \ref{hypotheses}, the random walk $X=(X_n)_{n\geq 1}$ admits a unique invariant probability measure
		\begin{align}
		m(dx)= &  \frac{\int_0^1 \kappa_{\theta} (x,v)dv}{\int_0^1 \int_0^1 \kappa_{\theta}(u,v) du\ dv} \ dx
		=\frac{\sum_{q=1}^Q \bar{\pi}_q \ind_{I_q}(x) \ dx}{\bar{\pi}}. \label{eq:pi}
		\end{align}
	\end{proposition}
	
	The general proof is given in \cite[Prop. 4.1]{athreyaroellin1} but for the case of SBM graphons, the result is easy to prove.\\
	
	\tvc{From expression \eqref{eq:pi}, we see that for $q\in \{1,\cdots Q\}$, the measure of the class $q$ with respect to $m(dx)$ is:
		\begin{equation}\label{def:alphatilde}\widetilde{\alpha}_q:=m(I_q)=\alpha_q \frac{\bar{\pi}_q}{\bar{\pi}}.\end{equation}
		So, if $\bar{\pi}_q>\bar{\pi}$, $\widetilde{\alpha}_q>\alpha_q$ and
		the stationary measure $m(dx)$ puts more weight on the interval $I_q$ which has a larger than average connectivity, compared with the Lebesgue measure. If $\bar{\pi}_1=\cdots \bar{\pi}_Q=\bar{\pi}$ are all equal, we have $\widetilde{\alpha}_q=\alpha_q$ for all $q\in \{1,\cdots Q\}$ and $m(dx)$ is the uniform measure on $[0,1]$ by \eqref{eq:pi}. Otherwise, we expect biases in how the graphon $\kappa_{\theta}$ is discovered by $G_n$.}
	
	\subsection{Convergence of dense graphs}
	
	We are interested in the case where $n\rightarrow +\infty$. Then, the (dense) RDS graph $G_n$ might converge to a graphon, and it is natural to compare the possible limit to the graphon $\kappa_{\theta}$ on which the random walk moves. Let us recall briefly some topological facts. We refer the interested reader to \cite{lovaszbook}.\\

	Let us give first some notations. For integers $n$ and $k\leq n$, $\lbrac 1,n\rbrac=\{1,2\cdots n\}$ and $(n)_k=n(n-1)\cdots (n-k+1)$.
	For a graph $G$, $E(G)$ denotes the edges of $G$ and $i\sim_G j$ means that $\{i,j\}\in E(G)$.
	We can define the subgraph $F$ density in $G$ by:
	\begin{equation}\label{eq:tFG}
	t(F,G)=\frac{\# \{\mbox{injections from }F\mbox{ to }G\}}{(n)_k}=\frac{1}{(n)_k}\sum_{(i_1,\cdots i_k)\in \lbrac 1,n\rbrac}\prod_{\{\ell,\ell'\}\in E(F)}\ind_{i_\ell \sim_G i_{\ell'}}
	\end{equation}where $\sum_{(i_1,\cdots i_k)\in \lbrac 1,n\rbrac}$ is a sum ranging over all vectors $(i_1,\cdots i_k)$ with mutually different coordinates in $\lbrac 1,n\rbrac$.
	This notion of subgraph density can be generalized to a graphon $\kappa$ by:
	\begin{equation}
	t(F,\kappa)= \int_{[0,1]^k} \prod_{\{\ell,\ell'\}\in E(F)}\kappa(x_\ell, x_{\ell'}) dx_1\cdots dx_k.
	\end{equation}
	Let $\mathcal{F}$ denote the class of isomorphism classes on finite graphs and let $(F_i)_{i\geq 1}$ be a particular enumeration of $\mathcal{F}$. Then, the distance of two graphs $G$ and $G'$ is:
	\begin{equation}\label{eq:dsub}
	d_{\sub}(G,G')=\sum_{i\geq 1} \frac{1}{2^i} \big|t(F_i,G)-t(F_i,G')\big|
	\end{equation}
	The convergence of the large graphs to graphons can be expressed with this distance \cite[Chapter 11]{lovaszbook}.

	\subsection{Biases in the discovery of $\kappa_{\theta}$}
	
	\tvc{	Let us denote by $\Gamma$ the cumulative distribution function of $m(dx)$:
		\begin{align}\label{def:Gamma}
		\Gamma(x)= &  \frac{\sum_{q=1}^Q \bar{\pi}_q \big[\min\big(\alpha_q, \ x-A_{q-1}\big)\big]_+}{\bar{\pi}}\nonumber\\
		= & \left\{\begin{array}{ll}
		\bar{\pi}_1 x & \mbox{ if }x\in I_1,\\
		\widetilde{A}_{q-1}+\bar{\pi}_q (x- \widetilde{A}_{q-1}) & \mbox{ if }x\in I_q,\\
		\end{array}\right.
		\end{align}where $\widetilde{A}_q=\sum_{k=1}^q \widetilde{\alpha}_k$. Notice that $\Gamma$ is a continuous piecewise affine function that maps $[A_{q-1},A_q)$ to $[\widetilde{A}_{q-1},\widetilde{A}_q)$.}

	Athreya and R\"ollin \cite{athreyaroellin1} have proved that the graphon discovered by the RDS is biased:	
	\begin{proposition}[Corollary 2.2 \cite{athreyaroellin1}] \label{th:athreya}We have under Assumptions \ref{hypotheses} that:
		\[\lim_{n\rightarrow +\infty} d_{\sub}\big(G_n,\kappa_{\Gamma^{-1}}\big)=0,\]
		where the generalised inverse of $\Gamma$ is
		\begin{align*}
		\Gamma^{-1} (v) =  \inf \{ u \in [0,1]: \Gamma(u) \geq v \},
		\end{align*}
		and where for all $x,y\in [0,1]$,
		\begin{equation}
		\kappa_{\Gamma^{-1}}(x,y)=\kappa\big(\Gamma^{-1}(x),\Gamma^{-1}(y)\big).
		\end{equation}
	\end{proposition}
	
	This proposition, that is true not only for SBM graphons but also in more general cases, as developed in \cite{athreyaroellin1}, says that the topology of the subgraph discovered by the RDS is biased compared with the true underlying structure ($\kappa$) because the random walk visits more likely the nodes with high degrees (hubs) and the frequent types.\\

	\tvc{In the case of an SBM graphon parameterized by $\theta=(\alpha_q, \pi_{qr} ; q,r\in \{1,\cdots Q\})$,
		and under Assumption \eqref{hypotheses}, $\Gamma$ is a one-to-one map and $\Gamma^{-1}$ is its usual inverse function: it is here the piecewise affine function that maps the interval $[\widetilde{A}_{q-1},\widetilde{A}_q)$ to $[A_{q-1},A_q)$. We have here:
		\begin{equation}
		\kappa_{\Gamma^{-1}}(x,y)=\kappa_{\widetilde{\theta}}(x,y),
		\end{equation}with the notation \eqref{grapho:SBM} and where
		\begin{equation}\widetilde{\theta}=(\widetilde{\alpha}_q ,\pi_{qr} ; q,r\in \{1,\cdots Q\}).
		\end{equation}
		For SBM graphons, there will be no bias when $\kappa_{\widetilde{\theta}}=\kappa_\theta$, \textit{i.e.} when for all $q\in \{1,\cdots Q\}$, $\widetilde{\alpha}_q=\alpha_q$.
	}
	\begin{example}
		When $Q = 2$, the graphon is given:
		\begin{align}
		\kappa_{\theta}(x,y) = \left\{ \begin{array}{ll} \pi_{11}, & 0 \leq x, y \leq \alpha;\\ \pi_{12}, & \alpha< x \leq 1 \quad \text{or} \quad \alpha < y \leq 1;\\
		\pi_{22}, & \text{otherwise}. \end{array} \right. \label{eq:2}
		\end{align}
		This function is represented in Fig. \ref{fig:graphonQ2}
		\begin{figure}[!ht]
			\centering
			\includegraphics[height=4.5cm,width=7cm,trim=0cm 1cm 0cm 1cm]{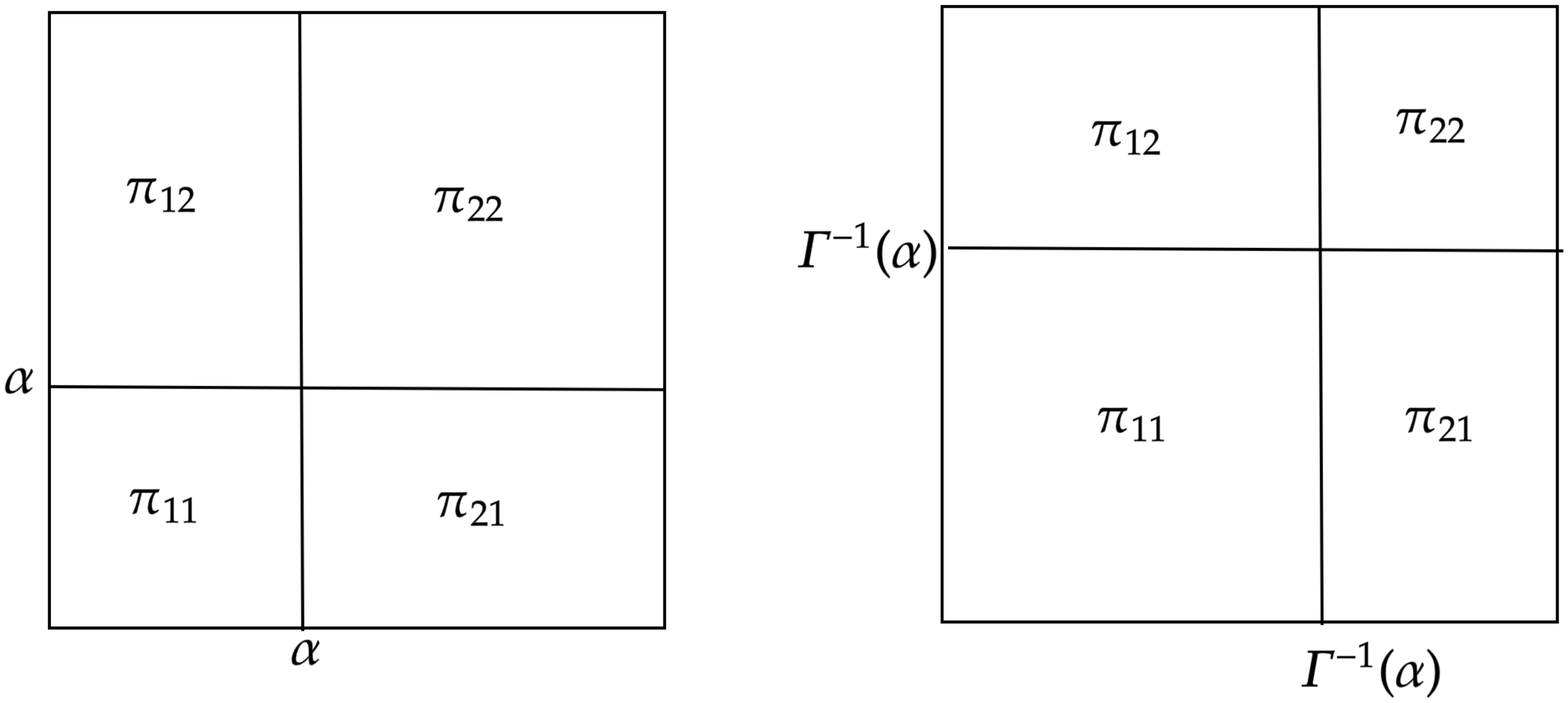}
			\caption{{\textit Left: Function $\kappa_{\theta}(x,y)$ for an SBM graphon with $Q=2$ classes. Right: Distorted graphon $\kappa_{\widetilde{\theta}}$ as discovered by the random walk. Notice that the parameters $\pi_{qr}$ are unchanged, but the weights of the classes are modified from $(\alpha,1-\alpha)$ to $(\Gamma(\alpha),1-\Gamma(\alpha))$. }}\label{fig:graphonQ2}
		\end{figure}
		
		The invariant probability measure is:
		\begin{align*}
		m(dx) = \frac{(\pi_{11}\alpha + \pi_{12}(1-\alpha)) \ind_{x \in [0, \alpha]}(x) + (\pi_{12}\alpha + \pi_{22} (1- \alpha)) \ind_{x \in (\alpha, 1]}(x)}{\pi_{11}\alpha^2 + 2\pi_{12} \alpha(1- \alpha) + \pi_{22} (1- \alpha)^2}dx.
		\end{align*}
		As a result (see Fig. \ref{fig:graphonQ2}), the bias graphon $\kappa_{\widetilde{\theta}}$ corresponds to the SBM graphon \eqref{eq:2} where the weights of the class 1 is changed from $\alpha$ to
		\begin{align}
		\Gamma(\alpha) = \frac{(\pi_{11}\alpha + \pi_{12}(1- \alpha))\alpha}{\pi_{11}\alpha^2 + 2\pi_{12} \alpha(1- \alpha) + \pi_{22} (1- \alpha)^2}.
		\end{align}
		In this particular case, it can be seen that $\Gamma(\alpha)=\alpha$ when $(1-\alpha)(\pi_{12}-\pi_{22})=\alpha (\pi_{12}-\pi_{11})$. This is satisfied for example when $\pi_{11}=\pi_{12}=\pi_{22}$ (Erd\"{o}s-Rényi) or when $\alpha=1/2$ and $\pi_{11}=\pi_{22}$ (both types are symmetric).
	\end{example}
	
	\subsection{Empirical cumulative distribution}

	As seen in the previous paragraph, the bias linked with the discovery of the graphon $\kappa_{\theta}$ by the RDS subgraph $G_n$ is expressed in term of the cumulative distribution $\Gamma$ of the stationary distribution $m$ of $X^{(n)}$. In the sequel, the empirical cumulative distribution of $m$ will be useful and we recall here some facts:
	\begin{equation}\label{eq:Gamman}
	\Gamma_n(x)=\frac{1}{n}\sum_{i=1}^n \ind_{X_i\leq x}\qquad \mbox{ and }\qquad \Gamma_n^{-1}(y)=\inf\big\{x\in [0,1]\ : \ \Gamma_n(x)\geq y\big\}.
	\end{equation}
	
	\begin{lemma}$\Gamma_n$ and $\Gamma_n^{-1}$ converge a.s. uniformly to $\Gamma$ and $\Gamma^{-1}$ respectively.
	\end{lemma}

	\begin{proof}The almost sure pointwise convergence of $\Gamma_n$ to $\Gamma$ is a consequence of the ergodic theorem. Then, the a.s. uniform convergence is obtain by the Glivenko-Cantelli theorem. \\
		Let us prove the uniform convergence of $\Gamma^{-1}_n$ to $\Gamma^{-1}$.
		Because all the $\alpha_q$'s are positive, $\Gamma$ is a nondecreasing and piecewise affine bijection and the inverse bijection $\Gamma^{-1}$ is also nondecreasing and piecewise affine.
		Let $\varepsilon>0$ and $n_0\in \N$ sufficiently large so that for all $n\geq n_0$, $\|\Gamma_n-\Gamma\|_\infty\leq \varepsilon$. Let $y\in [0,1]$. For $n\geq n_0$,
		\begin{align*}
		\big|\Gamma_n^{-1}(y)-\Gamma^{-1}(y)\big| \leq  & C \big|\Gamma(\Gamma_n^{-1}(y))-y\big|.
		\end{align*}Because the jumps of $\Gamma_n$ are a.s. of size $1/n$, we necessarily have that $y-\varepsilon\leq \Gamma(\Gamma_n^{-1}(y))\leq y+\varepsilon+\frac{1}{n}$. Thus,
		\begin{align*}
		\big|\Gamma_n^{-1}(y)-\Gamma^{-1}(y)\big| \leq  & C \big(\frac{1}{n}+\varepsilon\big),
		\end{align*}
		which proves the uniform convergence of $\Gamma_n^{-1}$ to $\Gamma^{-1}$.
	\end{proof}

	\section{Likelihood estimation}\label{sec:likelihood}
	
	In this section, we write the likelihood of $G_n$ and compute the MLE of the parameters $\theta$ in Section \ref{sec:completeobs}, when we have complete observations: $(Z_i, Y_{ij} ; i,j\in \{1,\cdots n\})$ are available. Here our likelihood is specific to the RDS exploration. The MLE does not have an explicit formula and we explain how to compute it numerically. Then in Section \ref{sec:likelihood-SAEM}, we study the case where the types $Z=(Z_1, \cdots, Z_n)$ of the nodes are unobserved. \\
	Notice that the estimation in this Section \ref{sec:likelihood} makes only use of the connectivity information carried by the random variables $Y_{ij}$. The estimators here do not depend on the positions $X_i$. The types $Z$ may be known or unobserved.  \\

	Let us introduce some notations. We define by $N^q_n$, $q \in \{1,..., Q\}$ the number of vertices of type $q$ sampled by the Markov chain. For $q, r\in \{1,...,Q\}$ we also define by:
	\begin{align*}
	& N_n^{q\leftrightarrow r}= \Card\big\{(i,j)\quad  | \quad   Z_i=q,\ Z_j=r, \ Y_{i,j}=1\big\};\\
	& N_n^{q\nleftrightarrow r}= \Card\big\{(i,j)\quad  | \quad  Z_i=q,\ Z_j=r,\ Y_{i,j}=0\big\}
	\end{align*}
	the number of couples of types $(q,r)$ that are connected (resp. not connected).

	\subsection{Complete observations} \label{sec:completeobs}
	
	Assume that we observe a subset of explored nodes discovered by the RDS, with their classes and connections: $(Z_i, Y_{ij} ;  i < j)\in \{1,\cdots Q\}^n\times \{0,1\}^{n(n-1)/2}$.\\

	\begin{proposition}
		Recall that $\theta=(\alpha_q,\pi_{qr} ; 1\leq q\leq r\leq Q)$. The complete likelihood of the observations is
\tvc{		\begin{align} \label{eq:likelihood-in-fctofN}
		\mathcal{L}(Z, Y, \theta) &= \prod_{1\leq q\leq r \leq Q}\pi_{qr}^{N^{q\leftrightarrow r}} (1-\pi_{qr})^{N^{q \nleftrightarrow r}}
		\times \prod_{q = 1}^{Q} \frac{\alpha_q^{N_n^q}}{(\sum_{q'=1}^{Q} \pi_{qq'} \alpha_{q'})^{N_n^q - \ind_{Z_n = q}}}.
		\end{align}}Notice that in the above formula, the notation $\pi_{qq'}$ is a shortcut for $\pi_{\min(q,q'),\max(q,q')}$.
	\end{proposition}
	
	\begin{proof}\tvc{We have that
		\begin{align*}
		\mathcal{L}(Z,Y ; \theta) &=\alpha_{Z_1} \prod_{m=1}^{n-1} \frac{\pi_{Z_m Z_{m+1}} \alpha_{Z_{m+1}}}{\sum_{q=1}^Q \pi_{Z_m q}\alpha_q} \times \prod_{\substack{1\leq i < j \leq n,\\
				|i-j|\not=1}} \pi_{Z_i Z_j}^{Y_{i,j}} (1-\pi_{Z_iZ_j})^{(1-Y_{i,j})},
		\end{align*}where the first product corresponds to the likelihood of the types sampled along the Markov chain, and the second product corresponds to the likelihood of edges between vertices that are not visited successively by the Markov chain. Because the graph is non-oriented, it is sufficient to consider $i<j$. Thus:
		\begin{align} \label{eq:likelihood}
		\mathcal{L}(Z,Y ; \theta) = \frac{ \prod_{i= 1}^{n} \alpha_{Z_i}}{ \prod_{i=1}^{n-1} \sum_{q = 1}^{Q} \pi_{Z_iq}\alpha_q} \times \prod_{1\leq i<j \leq n} b(Y_{ij}, \pi_{Z_iZ_j}),
		\end{align}
		where $b(Y_{ij}, \pi_{Z_iZ_j}) = \pi_{Z_iZ_j}^{Y_{ij}} (1- \pi_{Z_i Z_j})^{1- Y_{ij}}$ (recall that $Y_{i,i+1}=1$ by construction). Finally, rewriting the above likelihood using $N_n^q$, $N_n^{q \leftrightarrow r}$ and $N_n^{q \nleftrightarrow r}$, we obtain:
\begin{align}
		\mathcal{L}(Z, Y, \theta) &= \prod_{q=1}^{Q}\left(\frac{\pi_{qq}}{1- \pi_{qq}}\right)^{N_n^{q \leftrightarrow q}} (1-\pi_{qq})^{N_n^q(N_n^q -1)/2} \nonumber\\ & \quad \times \prod_{q < r} \left( \frac{\pi_{qr}}{1 - \pi_{qr}} \right)^{N_n^{q \leftrightarrow r}} (1- \pi_{qr})^{N_n^q N_n^r}
		\times \prod_{q = 1}^{Q} \frac{\alpha_q^{N_n^q}}{(\sum_{q'=1}^{Q} \pi_{qq'} \alpha_{q'})^{N_n^q - \ind_{Z_n = q}}},
		\end{align}which provides the announced result.
}
	\end{proof}	
	
	\begin{proposition}\label{prop:MLE}
		The MLE $\widehat{\theta} = (\widehat{\alpha}_q,\widehat{\pi}_{qr} ; 1\leq q\leq r \leq Q)$ is the solution of the following system of equations:
\begin{align}\label{eq:vraisemblance}
			& \frac{N_n^q}{\widehat{\alpha}_q} - \sum_{p=1}^{Q} \frac{(N_n^p - \ind_{Z_n=p}) \widehat{\pi}_{pq}}{\sum_{q'=1}^{Q} \widehat{\pi}_{pq'}\widehat{\alpha}_{q'}} = \frac{N_n^r}{\widehat{\alpha}_r} - \sum_{p=1}^{Q} \frac{(N_n^p - \ind_{Z_n=p}) \widehat{\pi}_{pr}}{\sum_{q'=1}^{Q} \widehat{\pi}_{pq'}\widehat{\alpha}_{q'}}   ;\\
			&   \frac{N_n^{q \leftrightarrow q}}{\widehat{\pi}_{qq}} - \frac{N_n^{q \nleftrightarrow q}}{1-\widehat{\pi}_{qq}} - \frac{(N_n^q - \ind_{Z_n=q})\widehat{\alpha}_q}{\sum_{q'=1}^{Q} \widehat{\pi}_{qq'} \widehat{\alpha}_{q'}}
	=0 ; \\
 &  \frac{N_n^{q \leftrightarrow r}}{\widehat{\pi}_{qr}} - \frac{N_n^{q \nleftrightarrow r}}{1-\widehat{\pi}_{qr}} - \frac{(N_n^q - \ind_{Z_n=q})\widehat{\alpha}_r}{\sum_{q'=1}^{Q} \widehat{\pi}_{qq'} \widehat{\alpha}_{q'}}
- \frac{(N_n^r - \ind_{Z_n=r})\widehat{\alpha}_q}{\sum_{q'=1}^{Q} \widehat{\pi}_{rq'} \widehat{\alpha}_{q'}}	=0 \quad \mbox{  if }q\not= r.
			\end{align}
	\end{proposition}
	
	\begin{proof}
		\thuy{		The log likelihood of the observations is:
		\begin{align*}
		\log \mathcal{L}=  &  \sum_{1\leq q\leq r\leq Q} N^{q\leftrightarrow r} \log(\pi_{qr})+N^{q\nleftrightarrow r} \log(1-\pi_{qr})\\
 & +\sum_{q=1}^Q \Big( N_n^q \log \alpha_q - (N_n^q-\ind_{Z_n=q})\log \big(\sum_{q'=1}^Q \pi_{qq'}\alpha_{q'}\big)\Big).
		\end{align*}
		When we optimize the function $\log \mathcal{L}$ with respect to the parameters and under the constraint that $\sum_{q=1}^Q \alpha_q=1$, we obtain after computation of the Lagrangian the following system. First, the estimator $\widehat{\theta} = (\widehat{\alpha}_q, \widehat{\pi}_{qr} ; 1\leq q\leq r\leq Q)$ satisfies the constraint
			\[\sum_{q=1}^Q \widehat{\alpha}_q=1.\]
		Second, the other equations of the system are:
			\begin{align*}
			\frac{\partial \log \mathcal{L}}{\partial \alpha_q} = \frac{\partial \log \mathcal{L}}{\partial \alpha_r} ; \quad
			\frac{\partial \log \mathcal{L}}{\partial \pi_{qr}} = 0.
			\end{align*}These equations give \eqref{eq:vraisemblance} for all $1\leq q\leq r\leq Q$. In the sequel, the example with $Q=2$ will be developed.				}
	\end{proof}

\thuy{The identifiability of the model where the sampling of nodes is i.i.d. is a result Allman et al. \cite[Theorem 7]{allmanmatiasrhodes}.
In our case, the consistence of $\widehat{\pi}_{qr}$ is obtained by Van der Vaart \cite[Th. 5.7]{vandervaart}.  Indeed, the sequence of log-likelihoods renormalized by $1/n^2$ converges to a limit when $n\rightarrow +\infty$ and this limit admits a local maximum around the true parameters $(\pi_{qr})$.
For the parameters $\alpha_q$, it is more tricky. Techniques developed by C\'elisse et al. \cite{celissdaudinpierre2012} and which are based on explicit expressions of the estimators can not be followed here. We can rewrite the likelihood of the $Y_{i,j}$'s as a mixture, given the probability of the $Z_i$'s, but the latter are not independent, which complicates the computation. This is left for further research.}


	\begin{remark}
		When the graph is completely observed and not only through the sampling from a Markov chain, the classical likelihood, as obtained in Daudin et al. \cite{daudinpicardrobin} is:
		\begin{align}\label{vraisemblance-classique}
		\mathcal{L}^{\class}(Z,Y ; \theta) =  & \prod_{i= 1}^{n} \alpha_{Z_i}  \times \prod_{1\leq i< j\leq n} b(Y_{ij}, \pi_{Z_iZ_j})\nonumber
		\\= & \prod_{q=1}^Q \alpha_q^{N_n^q}
		\times \prod_{1\leq q\leq r\leq Q}\pi_{qr}^{N_n^{q \leftrightarrow r}} (1-\pi_{qr})^{N_n^{q\nleftrightarrow r}}.
		\end{align}The difference between \eqref{vraisemblance-classique} and \eqref{eq:likelihood} is the first product which corresponds of the likelihood of the node types. In the classical case, these types are chosen independently whereas here they are discovered by the successive states of the Markov chain. In this classical case, the MLE has an explicit formula:
		\begin{equation}
		\widehat{\alpha}_q^{\class}=\frac{N_n^q}{n},\qquad \widehat{\pi}_{qr}^\class=\frac{N_n^{q\leftrightarrow r} }{N_n^q N_n^r},\qquad \widehat{\pi}_{qq}^\class=\frac{2N_n^{q\leftrightarrow q} }{N_n^q (N_n^q-1)}.\label{estimateurs:classiques}
		\end{equation}
	\end{remark}

	Here, for the likelihood \eqref{eq:likelihood-in-fctofN}, the MLE which solves \eqref{eq:vraisemblance} is not explicit any more. Let us discuss briefly the case of two classes ($Q=2$). The parameter is then $\theta=(\alpha,\pi_{11},\pi_{12},\pi_{22})$.
	Define $\widehat{\theta}= (\widehat{\alpha}, \widehat{\pi_{11}}, \widehat{\pi_{12}}, \widehat{\pi_{22}})$ the estimator of $\theta$.
The log likelihood is now:
\begin{align*}
		\log \mathcal{L}=  &   N^{1\leftrightarrow 1} \log(\pi_{11})+N^{1\nleftrightarrow 1} \log(1-\pi_{11})\\
 & +  N^{1\leftrightarrow 2} \log(\pi_{12})+N^{1\nleftrightarrow 2} \log(1-\pi_{12})\\
 & + N^{2\leftrightarrow 2} \log(\pi_{22})+N^{2\nleftrightarrow 2} \log(1-\pi_{22})\\
 & +  N_n^1 \log \alpha - (N_n^1-\ind_{Z_n=1})\log \big( \pi_{11}\alpha +\pi_{12}(1-\alpha)\big)\\
 & +  N_n^2 \log (1-\alpha) - (N_n^2-\ind_{Z_n=2})\log \big( \pi_{12}\alpha +\pi_{22}(1-\alpha)\big).
		\end{align*}Beware that the parameter $\pi_{12}$ appears in the two last lines. Then the estimators $\widehat{\theta}$ is the solution of
	\thuy{\begin{align}
		& \frac{N_n^1}{\widehat{\alpha}}- \frac{N_n^2}{1-\widehat{\alpha}}- \frac{(N_n^1 - \ind_{Z_n=1}) (\widehat{\pi_{11}}-\widehat{\pi_{12}})}{\widehat{\pi_{11}}\widehat{\alpha} + \widehat{\pi_{12}}(1-\widehat{\alpha})}- \frac{(N_n^2 - \ind_{Z_n=2}) (\widehat{\pi_{12}}- \widehat{\pi_{22}})}{\widehat{\pi_{12}}\widehat{\alpha} + \widehat{\pi_{22}}(1-\widehat{\alpha})} = 0 \label{the1steq};\\
		& \frac{N_n^{1\leftrightarrow1}}{\widehat{\pi_{11}}}
		- \frac{N_n^{1\nleftrightarrow 1}}{1-\widehat{\pi_{11}}}  - \frac{(N_n^1 - \ind_{Z_n=1}) \widehat{\alpha}}{\widehat{\pi_{11}}\widehat{\alpha}+\widehat{\pi_{12}}(1-\widehat{\alpha})}=0 \label{the2ndeq};\\
		& \frac{N_n^{1\leftrightarrow2}}{\widehat{\pi_{12}}}
		- \frac{N_n^{1\nleftrightarrow 2}}{1-\widehat{\pi_{12}}} - \frac{(N_n^1- \ind_{Z_n=1}) (1-\widehat{\alpha})}{\widehat{\pi_{11}}\widehat{\alpha}+\widehat{\pi_{12}}(1-\widehat{\alpha})} -\frac{(N_n^2 - \ind_{Z_n=2}) \widehat{\alpha}}{\widehat{\pi_{12}}\widehat{\alpha}+\widehat{\pi_{22}}(1-\widehat{\alpha})} = 0\label{the3rdeq};\\
		& \frac{N_n^{2\leftrightarrow2}}{\widehat{\pi_{22}}}
		- \frac{N_n^{2\nleftrightarrow 2}}{1-\widehat{\pi_{22}}} -\frac{(N_n^2 - \ind_{Z_n=2}) (1-\widehat{\alpha})}{\widehat{\pi_{12}}\widehat{\alpha}+\widehat{\pi_{22}}(1-\widehat{\alpha})}=0 \label{the4theq}.
		\end{align}
	}
	\thuy{Notice that the system of equations \eqref{the1steq}-\eqref{the4theq} is non-linear and can not be simplified further. Also, there does not exist the explicit solution for it. An algorithm for computing a particular solution for the case $Q=2$ is given in section 3.3.1 of the PhD thesis \cite{vo-these}. In our case, we use a numerical function: the \textbf{nlm} function of \textbf{R} to solve the system \eqref{the1steq}-\eqref{the4theq}  numerically to get the approximated values for the MLE $\widehat{\theta}$. }For the numerical simulations, we refer the reader to Section \ref{sec:numresults}.

	\subsection{Incomplete observations: SAEM Algorithm}\label{sec:likelihood-SAEM}


	Here, we assume that the types $Z=(Z_{i})_{i= 1,...,n}$ are unobserved.
	In this case, the likelihood of the observed data $Y=(Y_{ij} ;\ i,j\in \lbrac 1,n\rbrac)$ is obtained by
	summing the complete-data likelihood \eqref{eq:likelihood} over all the possible
	values of the unobserved variables $Z$:
	\begin{equation}\label{eq:vraisemblance-incomplete}
	\mathcal{L}(Y ; \theta) = \sum_{q_1,\cdots q_n=1}^Q \Big[ \frac{ \prod_{i= 1}^{n} \alpha_{q_i}}{ \prod_{i=1}^{n-1} \sum_{q = 1}^{Q} \pi_{q_iq}\alpha_q} \times \prod_{\substack{1\leq i<j\leq n \\ |i-j|\not= 1}} b(Y_{ij}, \pi_{q_iq_j})\Big],
	\end{equation}
	Unfortunately, this sum is not tractable and it is classical to use the Expectation-Maximization (EM) algorithm to compute the maximum likelihood. \tvc{Here we use an SAEM algorithm (see \cite{celeux:inria,kuhnlavielle}) with the variational approximation of the conditional distribution of $Z$ given $Y$ introduced in \cite{daudinpicardrobin}, and adapt their methods to our setting with the likelihood \eqref{eq:likelihood-in-fctofN}}.\\

	Let us sum up the EM algorithm (see e.g. \cite{celeux:inria,celeuxdiebolt,kuhnlavielle}).  Given the observed data: the Markov chain $X^{(n)}$, the connections $(Y_{ij},\ i,j\in X^{(n)})$ and the number of blocks $Q$ and the current estimator $\theta$, and given the value $\theta^{(k-1)}$ at the $(k-1)^{th}$ iteration of the EM, on the $k^{th}$ step, we compute the conditional expectation of the log-likelihood $\mathcal{L}(Z|X,Y,\theta^{(k)})$ given $X,Y$ for the current fit $\theta^{(k)}$. Here there is no explicit expression for the latter likelihood because the exact distribution of $Z$ given $X,Y$ is unknown and this we need to approximate it numerically by using an SAEM algorithm \cite{celeux:inria,kuhnlavielle}, proceeding as follows.
	
	\subsubsection{The SAEM algorithm}\label{sec:SAEM}
	Given the information of the $k-1$ iteration $\theta^{(k-1)} = (\alpha^{(k-1)}, \pi^{(k-1)})$, at the $k^{th}$ iteration of SAEM:
	\begin{itemize}
		\item[] \textbf{Step 1: Choosing the appropriate $Z^{(k)}$}
		\mbox{}\\
		- Simulate a candidate $Z^c$ following the proposal distribution $q_{\theta^{(k-1)}}(.| Z^{(k-1)})$. The choice of proposal distribution is discussed in Section \ref{sec:variational}, where we use a variational approach.
		\\
		- Calculate the acceptance probability
		\begin{equation}
		\omega(Z^{(k-1)}, Z^c) := \min\left\{ 1, \dfrac{\mathcal{L}(Z^c, Y, \theta^{(k-1)}) \cdot q_{\theta^{(k-1)}}(Z^{(k-1)}|Z^c)}{ \mathcal{L}(Z^{(k-1)}, Y, \theta^{(k-1)}) \cdot  q_{\theta^{(k-1)}}(Z^c|Z^{(k-1)}) } \right\};
		\end{equation}
		\\
		- Accept the candidate $Z^c$ with probability $\omega$: $\P(Z^{(k)} = Z^c) = \omega$ and $\P(Z^{(k)} = Z^{(k-1)}) = 1- \omega$.
		\\
		\item[] \textbf{Step 2: Stochastic approximation}
		Update the quantity
		\begin{align}
		& \mathcal{Q}^{(k)}(\theta) = \mathcal{Q}^{(k-1)}(\theta) + s_k \left( \log \mathcal{L}(Z^{(k)}, Y, \theta) - \mathcal{Q}^{(k-1)}(\theta) \right),
		\end{align}
		with the initialization $	\mathcal{Q}^{(0)}(\theta) := \mathbb{E}[\log \mathcal{L}(Z,Y, \theta^{(0)})]$ and $(s_k)_{k \in \mathbb{N}}$ is a positive decreasing step sizes sequence satisfying $\sum_{k =1}^{\infty} s_k = \infty$ and $\sum_{k =1}^{\infty} s_k^2 < \infty$.
		\\
		\item[] \textbf{Step 3: Maximization}
		Choose $\theta^{(k)}$ to be the value of $\theta$ that maximizes $\mathcal{Q}^{(k)}$	
		\begin{equation}
		\theta^{(k)} := \operatorname*{arg\,max}_{\theta}\mathcal{Q}^{(k)}(\theta).
		\end{equation}
	\end{itemize}
	
	Kuhn and Lavielle studied the convergence of the sequence $\theta^{(k)}$ in \cite{kuhnlavielle}. In the particular case of SBM, and for the incomplete likelihood based on \eqref{vraisemblance-classique}, the consistency of EM and variational methods has been studied by C\'elisse et al. \cite{celissdaudinpierre2012} and asymptotic normality has been established by Bickel et al. \cite{bickelchoichangzhang}. The likelihood that is considered here differs and these results can not be directly applied, but a study along these lines could be investigated.

	\subsubsection{Variational approach}\label{sec:variational}
	
	For the proposal distribution $q_{\theta^{(k-1)}}(. \ |\ Z^{(k-1)})$ of $Z^{(k)}$, we follow Daudin et al. \cite{daudinpicardrobin}, who use a variational approach. Let us recall the main idea of this approach. The general strategy has been described in Jordan
	et al. \cite{jordanghahramanijaakolasaul} or Jaakkola \cite{jaakola}.\\

	Recall the likelihood $\mathcal{L}(Y,\theta)$ of the incomplete data \eqref{eq:vraisemblance-incomplete}. The idea of the variational approach is to replace the likelihood by a lower bound:
	\begin{align}\label{def:JKullbackLeibler}
	\mathcal{J}(R_{Y,\theta}) = \log\mathcal{L}(Y,\theta) - \mathrm{KL}(R_{Y,\theta}(Z), \mathcal{L}(Z|Y,\theta)),
	\end{align}
	where $\mathrm{KL}(\mu, \nu) := \displaystyle\int d\mu \log\left( \dfrac{d\mu}{d\nu}\right)$ is the Kullback-Leibler divergence of distributions $\mu$ and $\nu$, and where $R_{Y,\theta}(Z)$ is an approximation of the conditional distribution $\mathcal{L}(Z|Y,\theta)$. When $R_{Y,\theta}$ is a good-approximation of $\mathcal{L}(Z|Y,\theta)$, $\mathcal{J}(R_{Y,\theta})$ is very closed to $\mathcal{L}(Y,\theta)$. \\
	Here, $Z$ takes discrete values in $\{1,..., Q\}$. Then,
	\begin{align}
	\mathcal{J}(R_{Y, \theta}) & = \log\mathcal{L}(Y, \theta) - \sum_{ (Z_1, ..., Z_n) \in \{ 1,..., Q\}^n} R_{Y, \theta}(Z) \log \frac{R_{Y, \theta}(Z)}{\mathcal{L}(Z|Y, \theta)}\nonumber\\
	& = \log \mathcal{L}(Y, \theta) - \sum_{ Z \in \{ 1,..., Q\}^n} R_{Y, \theta}(Z) \log R_{Y, \theta}(Z) \nonumber\\
	& \quad \quad \quad + \sum_{ Z \in \{ 1,..., Q\}^n} R_{Y, \theta}(Z) \log \mathcal{L}(Z|Y, \theta)\nonumber\\
	& = \log \mathcal{L}(Y, \theta) - \sum_{ Z \in \{ 1,..., Q\}^n} R_{Y, \theta}(Z) \log R_{Y, \theta}(Z) \nonumber\\
	& \quad + \sum_{ Z \in \{ 1,..., Q\}^n} R_{Y, \theta}(Z) \log \mathcal{L}(Z,Y,\theta)  - \sum_{ Z \in \{ 1,..., Q\}^n} R_{Y, \theta}(Z)\log \mathcal{L}(Y,\theta)\nonumber\\
	& = \sum_{ Z \in \{ 1,..., Q\}^n} R_{Y, \theta}(Z) \log \mathcal{L}(Z,Y,\theta) - \sum_{ Z \in \{ 1,..., Q\}^n} R_{Y, \theta}(Z) \log R_{Y, \theta}(Z)\nonumber\\
 & = \E_{R_{Y,\theta}}\big(\log \mathcal{L}(Z,Y,\theta)\big)-\E_{R_{Y,\theta}}\big(\log R_{Y, \theta}(Z)\big).\label{eq:J}
	\end{align}
	
	Following \cite{daudinpicardrobin}, we restrict to distributions $R_{Y,\theta}$ that belong to the family of multinomial probability distributions parameterized by $\tau=(\tau_1,\cdots \tau_Q)$, as approximated conditional distribution of $Z$ given $Y$ and $\theta$. \tvc{These multinomial distributions assume independence of the $Z_i$'s conditionally to the $Y$, which makes computations tractable}. If we look for the parameter $\tau$ that maximizes \eqref{def:JKullbackLeibler}, we will hence obtain the best approximation of $\mathcal{L}(Z|Y,\theta)$ among these multinomial distributions. We will chose the latter to be the proposal distribution for $Z$ in the Step 1 of the SAEM algorithm.	\\
	
	If $\ind_{Z_i}$ follows the multinomial distribution $\mathcal{M}(1; (\tau_{i1}, ..., \tau_{iq}))$, with
	$	\tau_{iq} = \P(Z_i = q|Y,\theta), $
	for $ i \in \{1, ..., n\}, q\in \{1,..., Q\} $, \tvc{and if the $Z_i$'s are independent with respect to $Y$,} then,
	\begin{equation}\label{eq:RYtheta}
	R_{Y, \theta}(Z) = \prod_{i =1}^{n} \tau_{i,Z_i}.
	\end{equation}

	We aim at calculating the parameter $\hat{\tau}$ that maximizes the lower bound of $\mathcal{L}(Y,\theta)$. Then the proposal distribution
	$q_{\theta^{(k-1)}}(.\ |\ Z^{(k-1)})$ for updating the types will be given by \eqref{eq:RYtheta} with the parameters $\widehat{\tau}$ given in the next proposition:

	\begin{proposition}
		Given $\alpha, \pi$, the optimal parameter \begin{align}
		\hat{\tau} &:= \operatorname*{arg\,max}_{\tau} \mathcal{J}(R_{Y, \theta}),
		\end{align}
		with constraint $\sum_{q=1}^Q \tau_{iq} =1, \forall i \in \{1,...,n\}$, satisfies the fixed point relation
		\begin{align}
		\tau_{iq} \propto \frac{\alpha_q}{\sum_{\ell=1}^Q \pi_{q\ell} \alpha_\ell} \prod_{i< j} \prod_{\ell=1}^Q b(Y_{ij}, \pi_{q\ell})^{\tau_{j\ell}}.
		\end{align}
	\end{proposition}
	
	\begin{proof}Using \eqref{eq:likelihood}, \eqref{eq:J} and \eqref{eq:RYtheta}, we have:
		\begin{align}
		\mathcal{J} (R_{Y, \theta}) = \sum_{i =1}^n \sum_{q=1}^Q \tau_{iq} \log \alpha_{q}   - \sum_{i=1}^{n-1} \sum_{q=1}^Q \log \left(  \sum_{r=1}^Q \pi_{qr} \alpha_r \right) \tau_{iq} \nonumber \\ + \sum_{i< j} \sum_{q,r =1}^Q \tau_{iq} \tau_{jr} \log b(Y_{ij}, \pi_{qr}) - \sum_{i=1}^{n} \sum_{q=1}^{Q} \tau_{iq} \log\tau_{iq}.
		\end{align}
		To solve the optimization problem $\operatorname*{arg\,max}_\tau \mathcal{J}(R_{Y, \theta})$ with constraint $\sum_{q=1}^Q \tau_{iq}=1$, we use the method of Lagrange multipliers, that is finding the optimal parameters $\tau, \lambda$ that maximize the Lagrangian function $\Lag(\tau, \lambda):= \mathcal{J}(R_{Y, \theta}) + \sum_{i=1}^n \lambda_i (\sum_{q=1}^Q \tau_{iq} -1)$, where $\lambda_i$ is the Lagrange multiplier.
		Take the derivative of $\mathcal{L}ag$ w.r.t. $\lambda_i$ and $\tau$, we have
		\begin{align*}
		\begin{cases}
		\dfrac{\partial \Lag }{\partial \lambda_i} = \displaystyle\sum_{q=1}^Q \tau_{iq} -1\\
		\dfrac{\partial \Lag}{\partial \tau_{iq}} = \log \alpha_q - \log \tau_{iq} + \lambda_i -1 - \log \displaystyle\sum_{r=1}^{Q} \pi_{qr} \alpha_r + \sum_{j>i} \sum _{r=1}^Q \tau_{jr} \log b(Y_{ij}, \pi_{qr}) \\ \hspace*{6cm}+ \displaystyle\sum_{j< i} \sum _{r=1}^Q \tau_{jr} \log b(Y_{ji}, \pi_{r q})
		\end{cases}.
		\end{align*}
		The optimal solution must satisfy $\displaystyle\frac{\partial \Lag }{\partial \lambda_i} = \frac{\partial \Lag}{\partial \tau_{iq}} = 0$, which implies
		\begin{align*}
		\log \tau_{iq} = \log \alpha_q + \lambda_i -1 - \log \sum_{r=1}^{Q} \pi_{qr} \alpha_r + \sum_{j\neq i} \sum _{r=1}^Q \tau_{jr} \log b(Y_{ij}, \pi_{qr}) .
		\end{align*}
		In other word,
		\begin{align}
		\tau_{iq} = e^{\lambda_i-1}\frac{\alpha_q}{\sum_{r=1}^Q \pi_{qr} \alpha_r} \prod_{i \neq j} \prod_{r=1}^Q b(Y_{ij}, \pi_{qr})^{\tau_{jr}}.
		\end{align}
	\end{proof}

	In the case $Q=2$, it turns out the problem is more simple since for each $i \in \{1,...,n\}, \tau_{i1} + \tau_{i2} = 1$. For sake of simplification, we denote by  $\tau_{i}$ instead of $\tau_{i1}$. Hence, $\tau_{i2} = 1 -\tau_{i1} = 1 - \tau_i$.
	\begin{proposition}\label{prop:tau-Q=2}
		When $Q=2$, the variational parameter $\tau_i $ has formula:
		\begin{align}
		\tau_{i} = \frac{\phi_i(\tau)}{ 1+ \phi_i(\tau)}=: \Phi_i(\tau),
		\end{align}
		where \begin{align}\label{def:phii}
		\phi_i(\tau ) := \frac{\alpha}{1- \alpha}   \frac{\alpha \pi_{21} + (1-\alpha) \pi_{22}}{ \alpha \pi_{11} + (1-\alpha) \pi_{12} }
		\prod_{j \neq i} \left( \frac{b(Y_{ij}, \pi_{12})}{b(Y_{ij}, \pi_{22})} \right)^{1/2} \nonumber\\ \times
		\prod_{j\neq i} \left(  \frac{b(Y_{ij}, \pi_{11}) b(Y_{ij}, \pi_{22})}{b(Y_{ij}, \pi_{12})^2} \right) ^ {\tau_{j}/2}.
		\end{align}	
	\end{proposition}
	
	\begin{proof}
		We solve directly the optimization problem $\max_{\tau} \mathcal{J}(R_{Y, \theta})$ without using the Lagrangian multiplier $\lambda$. The quantity $\mathcal{J}(R_{Y, \theta})$ is written explicitly as:
		\begin{align*}
		\mathcal{J}(R_{Y, \theta}) = \sum_{i=1}^n (\tau_i \log \alpha + (1- \tau_i) \log(1-\alpha) ) - \sum_{i=1}^{n} (\tau_i \log \tau_i + (1-\tau_i) \log(1-\tau_i) )\\
		+ \frac{1}{2} \sum_{i \neq j} \left[ \tau_i \tau_j \log b(Y_{ij}, \pi_{11}) + \tau_i(1-\tau_j) \log b(Y_{ij}, \pi_{12}) + (1-\tau_i) \tau_j \log b( Y_{ij}, \pi_{21} ) \right. \\
		+ \left.(1-\tau_i)(1- \tau_j) \log b(Y_{ij}, \pi_{22})  \right]
		- \sum_{i=1}^{n-1} [ \tau_i \log(\alpha \pi_{11} + (1- \alpha) \pi_{12}) \\
		+ (1-\tau_i) \log(  \alpha \pi_{21} + (1- \alpha) \pi_{22}].
		\end{align*}
		Take the derivative of $\mathcal{J}(R_{Y, \theta})$ w.r.t. $\tau_i$,
		\begin{align*}
		\frac{\partial \mathcal{J}}{\partial \tau_i} & =& \log\frac{\alpha}{1-\alpha}  + \log\frac{1- \tau_i}{\tau_i} + \frac{1}{2} \sum_{j\neq i} \left\{ \tau_j \log\frac{b(Y_{ij}, \pi_{11} )}{b(Y_{ij}, \pi_{21})} + (1- \tau_j) \log \frac{b( Y_{ij}, \pi_{12})}{b(Y_{ij}, \pi_{22})}  \right\}
		\\& &- \log \frac{\alpha \pi_{11} + (1- \alpha) \pi_{12} }{\alpha \pi_{21} + (1-\alpha) \pi_{22}}\\
		& =&\log\frac{\alpha}{1-\alpha}  - \log\frac{\tau_i}{1- \tau_i}  - \log \frac{\alpha \pi_{11} + (1- \alpha) \pi_{12} }{\alpha \pi_{21} + (1-\alpha) \pi_{22}} \hspace*{4cm}\\
		& &+ \frac{1}{2} \sum_{j\neq i} \tau_j \log \frac{b(Y_{ij}, \pi_{11}) b(Y_{ij}, \pi_{22}) }{ b(Y_{ij}, \pi_{12})^2 }
		+ \frac{1}{2} \sum_{j\neq i} \log \frac{b(Y_{ij}, \pi_{12}) }{ b(Y_{ij}, \pi_{22}) }.
		\end{align*}
		Then the variational parameter $\tau_i$ is the solution of equation $\frac{\partial \mathcal{J}}{\partial \tau_i} =0$, which gives
		\begin{align*}
		\frac{\tau_i}{1- \tau_i} = \frac{\alpha}{1-\alpha}  \times \frac{\alpha \pi_{11} + (1- \alpha) \pi_{12} }{\alpha \pi_{21} + (1-\alpha) \pi_{22}} \times 	\prod_{j \neq i} \left( \frac{b(Y_{ij}, \pi_{12})}{b(Y_{ij}, \pi_{22})} \right)^{1/2}
		\\\times \prod_{j\neq i} \left(  \frac{b(Y_{ij}, \pi_{11}) b(Y_{ij}, \pi_{22})}{b(Y_{ij}, \pi_{12})^2} \right) ^ {\tau_j/2}= \phi_i(\tau).
		\end{align*}
		It implies that $\tau_i = \frac{\phi_i(\tau)}{1 + \phi_i(\tau)} = \Phi_i(\tau)$.
	\end{proof}

	\subsubsection{Proposal distribution for the Step 1 of SAEM}
	
	For the sake of simplicity, we treat here the case $Q=2$, but generalization is straightforward. Using the previous results, we can now detail the Step 1 of the SAEM algorithm. Given the parameters $\theta^{(k-1)}$, the types $Z^{(k-1)}$ and the data $(Y_{ij} ; i,j \in \lbrac 1, n\rbrac)$, we proceed as follows.\\
	
	\noindent \textbf{Step 1:}	We compute the parameters $\tau_i^{(k)}$ as in Proposition \ref{prop:tau-Q=2}. The parameters in \eqref{def:phii} are given by $\theta^{(k-1)}$ and the terms
	$b(Y_{ij},\pi_{11}^{(k-1)})$, $b(Y_{ij},\pi_{12}^{(k-1)})$ and $b(Y_{ij},\pi_{22}^{(k-1)})$ are computed with the types $Z^{(k-1)}$.\\

	\noindent \textbf{Step 2:}	We simulate a candidate $Z^c \in \{1,2\}^n$ for $Z$ such that $Z_i^c - 1$ follows the law $ \mathcal{B}er(\tau_i)$.
	Recall that the acceptance probability is
	
	\begin{equation}
	\mu(Z^{(k-1)}, Z^c) := \min\left\{ 1, \dfrac{\mathcal{L}_{\text{com}}(Z^c, Y, \theta^{(k-1)}) q_{\theta^{(k-1)}}(Z^{(k-1)}|Z^c)}{ \mathcal{L}_{\text{com}}(Z^{(k-1)}, Y,\theta^{(k-1)}) q_{\theta^{(k-1)}}(Z^c|Z^{(k-1)})} \right\},
	\end{equation}
	where the complete likelihood with respect to $\alpha, \pi, Z, Y$ is
	\begin{align*}
	\mathcal{L}_{\text{com}}(Z, Y, \theta) &= \prod_{q=1}^{Q}\left(\frac{\pi_{qq}}{1- \pi_{qq}}\right)^{N_n^{q \leftrightarrow q}} (1-\pi_{qq})^{N_n^q(N_n^q -1)/2} \\ & \quad \times \prod_{q \neq r} \left( \frac{\pi_{qr}}{1 - \pi_{qr}} \right)^{N_n^{q \leftrightarrow r}} (1- \pi_{qr})^{N_n^q N_n^r}
	\prod_{q = 1}^{Q} \frac{\alpha_q^{N_n^q}}{(\sum_{q'=1}^{Q} \pi_{qq'} \alpha_{q'})^{N_n^q - \ind_{Z_n = q}}}.
	\end{align*}
	and
	\begin{align*}
	q_{\theta^{(k-1)}} (Z^c|Z^{(k-1)})  &=\prod_{i=1} \tau_i^{2-Z_i^c} (1-\tau_i)^{Z_i^c -1}; \\
	\quad q_{\theta^{(k-1)}} (Z^{(k-1)}| Z^c) &=  \prod_{i=1} \tau_i^{2-Z_i^{(k-1)}} (1-\tau_i)^{Z_i^{(k-1)} -1}.
	\end{align*}\\

	\section{Estimation via biased graphon and `classical likelihood'}\label{sec:graphon}

	In Section \ref{sec:likelihood}, the MLE are computed but they do not have explicit formula in the case of RDS exploration. We thus investigate other estimators. The most natural one is the graphon estimator corresponding to \eqref{estimateurs:classiques}. It turns out that we can study the asymptotic bias of this estimator thanks to the result of Athreya and R\"ollin \cite{athreyaroellin1}.
	\tvc{First, in Section \ref{sec:completeobs2} we provide a two-step estimator in the case where everything is observed: $(X_i, Z_i, Y_{ij} ; i,j\in \{1,\cdots n\})$ are available. This new estimator is explicit: we compute the estimator \eqref{estimateurs:classiques} of Daudin et al. \cite{daudinpicardrobin} and then correct the weights of classes according to the formula of Athreya and R\"ollin (see \eqref{eq:widehatalpha}).}\\
	\tvc{Then in Section \ref{sec:incomplete_graphon}, when the $Z_i$'s are unobserved, we propose an SAEM estimator based on the one introduced above. Here, we need some to have the knowledge on the positions $X_i$'s of the Markov chain $X^{(n)}$ when the $Z_i$'s are missing. Notice however that (i) the knowledge of the $X_i$'s gives partial knowledge on the types $Z_i$'s since the latter are determined from the $X_i$'s once the intervals $I_q$ are given and (ii) the likelihood function \eqref{eq:likelihood-in-fctofN} depends on the $X_i$'s only through the $Z_i$'s}.

	\subsection{Complete observations}\label{sec:completeobs2}

	Assume in this section that we observe $X^{(n)}=(X_1,\dots X_n)$, the types $(Z_i)_{i\in \{1,\dots n\}}$ and the adjacency matrix $(Y_{ij})_{i, j\in \{1,\dots n\}}$ of the subgraph $G_n=G(X^{(n)},\kappa,H_n)$. \\
	
	From the result of Athreya and R\"ollin \cite{athreyaroellin1}, $G_n$ converges to the SBM graphon $\kappa_{\widetilde{\theta}}$ of parameter $\widetilde{\theta}= (\widetilde{\alpha}_q, \pi_{qr} ; q,r\in \{1,\cdots Q\} )$. This leads to a natural two-stages estimation of the parameter $\theta$ that we now define.\\

\begin{definition}The estimator of $\theta$, is defined in two steps.\\

\noindent \textbf{First step:} we estimate $\widetilde{\theta}= (\widetilde{\alpha}, \pi)$. A natural estimator is the classical MLE when assuming that there is no biases. Let us therefore define:
\begin{align}
		\widehat{\lambda}^n_q := \frac{N_n^q}{n}; \quad \widehat{\pi}_{qr}^n := \frac{N_n^{q\leftrightarrow r}}{N_n^q N_n^r} \quad \text{for} \quad q \neq r \quad \text{and} \quad \widehat{\pi}_{qq}^n := \frac{2N^{q\leftrightarrow q}_n}{N_n^q(N_n^q - 1)}.
		\end{align}

\noindent \textbf{Second step:} we correct the estimator $\widetilde{\theta}$ to obtain $\theta$. Especially, we specify an estimator of $\alpha_q$ obtained by correcting the estimator $\widehat{\lambda}_q$ of $\widetilde{\alpha}_q$. For this, we set for $q\in \{1,\dots Q\}$, $\widehat{\Lambda}^n_q=\sum_{k=1}^q \widehat{\lambda}^n_k$ and define
\begin{equation}\label{eq:widehatalpha}
		\widehat{\alpha}_q^n=\Gamma_n^{-1}\big(\widehat{\Lambda}^n_q\big)-\Gamma_n^{-1}\big(\widehat{\Lambda}^n_{q-1}\big),
\end{equation}where $\Gamma_n$ is the cumulative empirical distribution function of the $X_i$'s, see \eqref{eq:Gamman}.\\

Let us define by $\widehat{\theta}=(\widehat{\alpha}^n_q , \widehat{\pi}^n_{qr} ; q,r\in \{1,\dots Q\})$ the estimator of $\theta$.
\end{definition}

To understand \eqref{eq:widehatalpha}, recall that from \eqref{def:Iq} and \eqref{def:Gamma}:
\begin{equation}
\alpha_q= A_q-A_{q-1}=\Gamma^{-1}\big(\widetilde{A}_q\big)-\Gamma^{-1}\big(\widetilde{A}_{q-1}\big).
\end{equation}where $\widetilde{A}_q$ are defined under Equation \eqref{def:Gamma}.



\begin{proposition}\label{eq:estimatorsAR}Under Assumptions \ref{hypotheses},\\
		(i) For all $q,r\in \{1,\cdots Q\}$ $\widehat{\lambda}^n_q$ is a consistent estimator of $\widetilde{\alpha}_q$ and $\widehat{\pi}^n_{qr}$ is a consistent estimator of $\pi_{qr}$:
		\begin{equation}\label{eq:cv_rho_gamma}\lim_{n\rightarrow +\infty} \widehat{\pi}^n_{qr}=\pi_{qr},\qquad \mbox{ and }\qquad \lim_{n\rightarrow +\infty}\widehat{\lambda}_q^n=\Gamma(A_q)-\Gamma(A_{q-1}) = \widetilde{\alpha}_q,\end{equation}where we recall the notations of \eqref{def:Iq} and \eqref{def:alphatilde}.\\
		(ii) It follows that $\widehat{\alpha}_q^n$ is a consistent estimator of $\alpha_q$ for all $q\in \{1,\cdots Q\}$: almost surely,
\[\lim_{n\rightarrow +\infty}\widehat{\alpha}_q^n=\alpha_q.\]
		In the special case of $Q=2$, an estimator of $\alpha_1$ is $ \widehat{\alpha}^n_1= \Gamma_n^{-1}(\widehat{\lambda}^n_1)$.
	\end{proposition}	

The proof of Proposition \ref{eq:estimatorsAR} is done in the next section (Section \ref{sec:proof_estimatorAR}).\\

	\tvc{	We can go a little further: we indeed have two empirical approximations of the limiting graphon $\kappa_{\widetilde{\theta}}$: the graph $G_n$ (which converge to $\kappa_{\widetilde{\theta}}$ by the result of Athreya and R\"ollin) and the graphon $\widehat{\chi}_n$ associated with $\widehat{\theta}$ and defined below (whose convergence remains to be proved). The following result concludes that these two approximations are asymptotically equal, providing as a result the convergence of $\widehat{\chi}_n$. It is proved in Section \ref{sec:limit_chin}.}

	\begin{proposition}\label{prop:GnChin}The graphon associated to the estimator $(\widehat{\lambda}^n_q, \widehat{\pi}_{qr} ; q,r\in \{1,\dots Q\})$ is defined as:
		\begin{align}
		\widehat{\chi}_n(x,y) := \sum_{q=1}^Q\sum_{r=1}^Q \widehat{\pi}_{qr}^n \ind_{J^n_q}(x)\ind_{J^n_r}(y),
		\end{align}
		with $J^n_q = [ \widehat{\Lambda}^n_{q-1},\widehat{\Lambda}^n_q)$ and $\widehat{\Lambda}^n_q$ are defined above \eqref{eq:widehatalpha}.
We have under Assumption \ref{hypotheses} that:\\
		(i) when $n\rightarrow +\infty$,
		\begin{align}
		\lim_{n\rightarrow +\infty} d_{sub}(G_n, \widehat{\chi}_n) =0.
		\end{align}
		(ii) The limit of the empirical graphon $\widehat{\chi}_n$ is thus the biased graphon $\kappa_{\widehat{\theta}}$.
		\begin{align}
		\lim_{n \rightarrow +\infty} d_{\sub}(\widehat{\chi}_n, \kappa_{\widehat{\theta}}) = 0.
		\end{align}
	\end{proposition}


\subsubsection{Proof of Proposition \ref{eq:estimatorsAR}}	\label{sec:proof_estimatorAR}

		Let us consider point (i) of Proposition \ref{eq:estimatorsAR}. The limit for $\widehat{\lambda}_q^n$ follows from the ergodic theorem. Indeed, we can write that
		\begin{align*}
		\widehat{\lambda}_q^n = \frac{N_n^q}{n} = \frac{1}{n} \sum_{i = 1}^n \ind_{X_i^{(n)} \in I_q}.
		\end{align*}
		The ergodic theorem for the Markov chain $(X^{n})_n$ says that
		\begin{align*}
		\lim\limits_{n \rightarrow +\infty}	\frac{1}{n} \sum_{i = 1}^n \ind_{X_i^{(n)} \in I_q} = \E_m[ \ind_{ X_1 \in I_q }] = \Gamma(A_q) - \Gamma(A_{q-1}) = \widetilde{\alpha}_q.
		\end{align*}\\
		
		It remains to prove that $\widehat{\pi}_{qr}^n$ is a consistent estimator of $\pi_{qr}$. Rewrite $\widehat{\pi}_{qr}^n$  as
		\[ \widehat{\pi}_{qr}^n = \frac{N^{q \leftrightarrow r}_n/n^2}{\frac{N^{q}_n}{n} \frac{N^{r}_n}{n}} = \frac{1}{\widehat{\lambda}_q^n \widehat{\lambda}_r^n} \frac{1}{n^2} N_n^{q \leftrightarrow r}.\]
		Recall that the subgraph $G_n$ is constructed from the Markov chain $X^{(n)}$ and that each pair of non-consecutive vertices $X_i$ and $X_j$ are connected with probability $\kappa_{\theta}(Z_i, Z_j)$ depending on theirs types and independently of the others edges. Let us focus on the number of edges $N_n^{q\leftrightarrow r}$: two cases have to be distinguished.\\

		\noindent \textbf{Case 1, $q\not=r$:} The number of edges of types $(q,r)$ is
		\[ N^{q \leftrightarrow r}_n = \sum_{i=1}^{n-1} \ind_{X_i \in I_q, X_{i+1} \in I_r} + \sum_{\substack{1\leq i,j\leq n\\ \|i-j\| \neq 1 }}\ind_{i \sim_{G_n} j} \ind_{X_i \in I_q, X_j \in I_r}.\]
		Then,
		\begin{equation}\label{etape1} \widehat{\pi}^n_{qr} = \frac{1}{\widehat{\lambda}_q^n \widehat{\lambda}_r^n n} \left(\frac{1}{n}\sum_{i=1}^{n-1} \ind_{X_i \in I_q, X_{i+1} \in I_r} \right) + \frac{1}{n^2} \sum_{\substack{1\leq i,j\leq n\\ \|i-j\| \neq 1 }}\frac{\ind_{i \sim_{G_n} j} \ind_{X_i \in I_q, X_{j} \in I_r} }{\widehat{\lambda}_q^n \widehat{\lambda}_r^n }. \end{equation}
		By the ergodic theorem for Markov chain $X^{(n)}$, we have
		\[\lim_{n \rightarrow +\infty} \frac{1}{n}\sum_{i=1}^{n-1} \ind_{X_i \in  I_q, X_{i+1} \in I_r} = \E_m[ \ind_{X_0 \in I_q, X_1 \in I_r}] = \widetilde{\alpha}_q \pi_{qr}< +\infty.\] Since $\lim_{n \rightarrow +\infty}\widehat{\lambda}_q^n = \widetilde{\alpha}_q>0$ in probability, there exists a constant $c>0$ such that $c\leq \inf_{q\in \{1,\dots Q\}} \widetilde{\alpha}_q$ and
		\[ \lim_{n\rightarrow +\infty} \P\left(\frac{1}{\widehat{\lambda}_q^n \widehat{\lambda}_r^n n} \left(\frac{1}{n}\sum_{i=1}^{n-1} \ind_{X_i \in I_q, X_{i+1} \in I_r} \right) \leq \frac{1}{c^2 n} \left(\frac{1}{n}\sum_{i=1}^{n-1} \ind_{X_i \in I_q, X_{i+1} \in I_r} \right)\right)  = 1,\] and hence the first term in the right hand side of \eqref{etape1} converges to 0 in probability. \\
		
		Consider now the second term in the r.h.s. of \eqref{etape1}. Let us define the function
		\[f (G_n)= \frac{1}{n^2} \sum_{\substack{1\leq i,j\leq n\\ \|i-j\| \neq 1}}\ind_{i \sim_{G_n} j} \ind_{X_i \in I_q, X_{j} \in I_r},\] then $f$ is a function of the $n(n-1)/2 - (n-1) = (n-1)(n-2)/2$ random edges on $n$ vertices.
		We see that
		\[\E[f(G_n)]=\E\Bigg[ \frac{1}{n^2} \sum_{\substack{1\leq i,j\leq n\\ \|i-j\| \neq 1}}\ind_{i \sim_{G_n} j} \ind_{X_i \in I_q, X_{j} \in I_r} \Bigg] = \frac{(n-1)(n-2)}{n^2} \pi_{qr} \widetilde{\alpha}_q \widetilde{\alpha}_r .\]
		We have
		\begin{multline*}
		\P\left( \left| \frac{1}{n^2} \sum_{\substack{1\leq i,j\leq n \\ \|i-j\| \neq 1}}\frac{\ind_{i \sim_{G_n} j} \ind_{X_i \in I_q, X_{j} \in I_r} }{\widehat{\lambda}_q^n \widehat{\lambda}_r^n }  - \pi_{qr}\right| > \varepsilon\right)\\
		\begin{aligned}
		\leq &  \P\left( \frac{1}{\widehat{\lambda}_q^n \widehat{\lambda}_r^n} \big|  f(G_n) - \E[f(G_n)] \big| > \varepsilon - \left| \frac{1}{\widehat{\lambda}_q^n \widehat{\lambda}_r^n} \E[f(G_n)] - \pi_{qr} \right| \right) \nonumber\\
		= &  \P\left( \big|  f(G_n) - \E[f(G_n)] \big| > \varepsilon \widehat{\lambda}_q^n \widehat{\lambda}_r^n - |\E[f(G_n)] - \widehat{\lambda}_q^n \widehat{\lambda}_r^n \pi_{qr}| \right) \nonumber\\
		= & \P \left( \big|  f(G_n) - \E[f(G_n)] \big| > \varepsilon \widehat{\lambda}_q^n \widehat{\lambda}_r^n - \pi_{qr} \left| \frac{(n-1)(n-2)}{n^2} \widetilde{\alpha}_q \widetilde{\alpha}_r - \widehat{\lambda}_q^n \widehat{\lambda}_r^n \right| \right)
		\end{aligned}
		\end{multline*}
		For $c < \inf_{q\in \{1,\dots Q\}} \widetilde{\alpha}_q$,
		\begin{multline}\label{eq:macdiarmid1}
		\P \left( \big|  f(G_n) - \E[f(G_n)] \big| > \varepsilon \widehat{\lambda}_q^n \widehat{\lambda}_r^n - \pi_{qr} \left| \frac{(n-1)(n-2)}{n^2} \widetilde{\alpha}_q \widetilde{\alpha}_r - \widehat{\lambda}_q^n \widehat{\lambda}_r^n \right| \right) \\
		\leq \P \left( \big|  f(G_n) - \E[f(G_n)] \big| > c^2\varepsilon - \frac{c^3}{2} \varepsilon\right) \hspace*{5cm}
		\\+ \P\left(\left| \frac{(n-1)(n-2)}{n^2} \widetilde{\alpha}_q \widetilde{\alpha}_r - \widehat{\lambda}_q^n \widehat{\lambda}_r^n \right| > \frac{c^3\varepsilon}{2\pi_{qr}}\right) + \P(\widehat{\lambda}_q^n \widehat{\lambda}_r^n < c^2).
		\end{multline}
		Since $\lim_{n \rightarrow +\infty}\widehat{\lambda}_q^n = \widetilde{\alpha}_q>0$ in probability, for fixed $\varepsilon>0$,
		\[
		\lim_{n \rightarrow \infty}\P \left( \left| \frac{(n-1)(n-2)}{n^2} \widetilde{\alpha}_q \widetilde{\alpha}_r - \widehat{\lambda}_q^n \widehat{\lambda}_r^n \right| < \frac{c^3\varepsilon}{2\pi_{qr}} \ \text{ and } \  \widehat{\lambda}_q^n \widehat{\lambda}_r^n > c^2 \right) =1
		\] Thus the second and the third terms on the right hand side of \eqref{eq:macdiarmid1} tend to zero as $n$ tends to infinity. It remains the first term to be treated. When one edge is changed, the value of $f$ is changed by most $1/n^2$. Applying McDiarmid's concentration \cite{mcdiarmid} for function $f$ , we obtain:
		\begin{align*}
		\P \left( \big|  f(G_n) - \E[f(G_n)] \big| > c^2\varepsilon - \frac{c^3}{2} \varepsilon\right)	\leq 2\exp\left( - \frac{2(c^2 -\frac{c^3}{2})\varepsilon } {\frac{(n-1)(n-2)}{2} \frac{1}{n^4}  } \right)
		\leq 2e^{  -4n^2 c^2(1-c/2) \varepsilon}.
		\end{align*}
		Note that $0<c<1$ then $c^2 (1-c/2)>0$. We use Borel-Cantelli's Theorem to conclude that
		$ \lim_{n \rightarrow +\infty} \P \left( \big|  f(G_n) - \E[f(G_n)] \big| > c^2\varepsilon - \frac{c^3}{2} \varepsilon\right) =0$ and hence,
		\[  \left| \frac{1}{n^2} \sum_{\substack{1\leq i,j\leq n\\ \|i-j\| \neq 1}}\frac{\ind_{i \sim_{G_n} j} \ind_{X_i \in I_q, X_{j} \in I_r} }{\widehat{\lambda}_q^n \widehat{\lambda}_r^n }  - \pi_{qr}\right| \longrightarrow 0\]
		in probability as $n \rightarrow \infty$. This finishes the proof for Case 1.\\
		
		\noindent \textbf{Case 2, $q=r$:} The proof follows by similar arguments, with notice that there are a few modifications because the expression of $N_n^{q \leftrightarrow q}$ is slightly different:
		\[ N^{q \leftrightarrow q}_n = \sum_{i=1}^{n-1} \ind_{X_i \in I_q, X_{i+1} \in I_q} + \frac{1}{2} \sum_{\substack{1\leq i,j\leq n\\ \|i-j\| \neq 1}}\ind_{i \sim_{G_n} j} \ind_{X_i \in I_q, X_j \in I_q}.\] Then,
		\begin{equation}\label{etape2} \widehat{\pi}^n_{qq} = \frac{1}{\widehat{\lambda}_q^n \big(n\widehat{\lambda}_q^n -1\big)} \left(\frac{1}{n}\sum_{i=1}^{n-1} \ind_{X_i \in I_q, X_{i+1} \in I_q} \right) + \frac{1}{n^2} \sum_{\substack{1\leq i,j\leq n\\ \|i-j\| \neq 1}}\frac{\ind_{i \sim_{G_n} j} \ind_{X_i \in I_q, X_{j} \in I_q} }{\widehat{\lambda}_q^n\big(\widehat{\lambda}_q^n -1/n \big) } \end{equation}
		We have that the first term on r.h.s. of \eqref{etape2} converges in probability to $0$ as in case 1. For the second term on r.h.s. of \eqref{etape2}, we define the function $f$ as in Case 1 by
		\[f (G_n)= \frac{1}{2n^2} \sum_{\substack{1\leq i,j\leq n\\ \|i-j\| \neq 1}}\ind_{i \sim_{G_n} j} \ind_{X_i \in I_q, X_{j} \in I_q},\]
		For a fixed $\varepsilon>0$,
		\begin{align*}
		\P &\left( \big| \frac{1}{n^2} \sum_{\substack{1\leq i,j\leq n\\ \|i-j\| \neq 1}}\frac{\ind_{i \sim_{G_n} j} \ind_{X_i \in I_q, X_{j} \in I_q} }{\widehat{\lambda}_q^n\big(\widehat{\lambda}_q^n -1/n \big) }  - \pi_{qq} \big| > \varepsilon \right) \\
		&\leq \P \left( \big|  f(G_n) - \E[f(G_n)] \big| > \varepsilon \widehat{\lambda}_q^n \big(\widehat{\lambda}_q^n -1/n \big) \right. \\ \ & \left.  \hspace{3cm}- \pi_{qq} \left| \frac{(n-1)(n-2)}{n^2} (\widetilde{\alpha}_q)^2 - \widehat{\lambda}_q^n \big(\widehat{\lambda}_q^n-1/n\big) \right| \right)\\
		&\leq \P \left( \big|  f(G_n) - \E[f(G_n)] \big| > c\big(c-\frac{1}{n}\big)\varepsilon - \frac{c^3}{2} \varepsilon\right) + \P(\widehat{\lambda}_q^n < c) \\
		&\hspace{2cm} + \P\left(\left| \frac{(n-1)(n-2)}{n^2} (\widetilde{\alpha}_q)^2 - \widehat{\lambda}_q^n\big(\widehat{\lambda}_q^n-\frac{1}{n}\big) \right| > \frac{c^3\varepsilon}{2\pi_{qq}}\right).
		\end{align*}
		As in Case 1, the second and the third term on r.h.s. of above inequality are negligible. Applying McDiarmid's concentration for $f$ with notice that when changing 1 edge in $G_n$, the value of $f$ changes at most $1/n^2$,
		\begin{align*}
		\P \left( \big|  f(G_n) - \E[f(G_n)] \big| > c(c-1/n)\varepsilon - \frac{c^3}{2} \varepsilon\right)	&\leq 2\exp\left( - \frac{2(c^2- c/n -\frac{c^3}{2})\varepsilon } {\frac{(n-1)(n-2)}{2} \frac{1}{n^4}  } \right)\\
		&\leq 2e^{  -2(n^2 c^2(1-c/2) -nc) \varepsilon}.
		\end{align*}
		Finally, using Borel-Cantelli's Theorem, $|f(G_n) - \E[f(G_n)]|\rightarrow 0$ almost surely as $n$ tends to infinity. Thus, the point (i) is proved.


\subsubsection{Proof of Proposition \ref{prop:GnChin}: Limit of $\widehat{\chi}_n$}\label{sec:limit_chin}
	
		For the point (ii), we have:
		\begin{align*}
		d_{\sub}(\widehat{\chi}_n, \kappa_{\widehat{\theta}})\leq  & d_{\sub}(\widehat{\chi}_n, G_n)+d_{\sub}(G_n, \kappa_{\widehat{\theta}}).
		\end{align*}The first term in the right hand side is treated by point (i). The second term is the Proposition \ref{th:athreya} shown in \cite[Corollary 2.2]{athreyaroellin1}.\\

Let us now consider the point (i). For the sake of simplicity, we assume for the proof that there are two classes of vertices in the graph, i.e. $Q=2$. The proof can be generalized to general $Q$ by following the same steps. Our parameters' notations are simplified as $\lambda_1^n =: \lambda_n$ and $\lim_{n \rightarrow +\infty}\lambda^n_1=: \widetilde{\alpha} = \Gamma(\alpha)$.\\
	
	Our purpose is to prove a convergence of graphons for the distance $d_{sub}$ introduced in \eqref{eq:dsub} using the densities \eqref{eq:tFG}. If $F$ is an edge (meaning that $F = K_2$, the complete graph of $2$ vertices), then the density of $F$ in $G_n := G(X_n, H_n, \kappa)$ is the proportion of edges,
	\begin{align*}
	t(F, G_n) & = \frac{1}{n(n-1)} \sum_{\ell,\ell'\in \lbrac 1,n\rbrac} \ind_{\ell\sim_{G_n} \ell'}\\
	\text{and} \quad t(F, \chi_n) & = \int\limits_{[0,1]^2} \widehat{\chi}_n(x_1,x_2)dx_1dx_2= \sum_{q, r =1}^Q\widehat{\lambda}_q^n\widehat{\lambda}_r^n \widehat{\pi}_{qr}^n .
	\end{align*}
	%
	In general case, if $F$ is a graph of $k$ vertices,
	\begin{align}
	t(F,G_n) &=\frac{1}{(n)_k}\sum_{(i_1,\cdots i_k)\in \lbrac 1,n\rbrac}\prod_{\{\ell,\ell'\}\in E(F)}\ind_{i_\ell \sim_G i_{\ell'}}\\
	t(F,\chi_n)& = \int_{[0,1]^k} \prod_{\{\ell,\ell'\}\in E(F)} \left(\sum_{q,r =1}^Q\widehat{\pi}^{qr}_n \ind_{J^n_q\times J^n_r}(x_\ell,x_{\ell'}) \right) dx_1\cdots dx_k
	\end{align}
	
	Let us first consider the case where $F$ is an edge.
	\begin{multline*}
	|t(F, G_n) - t(F, \chi_n)|  =  \left|\frac{1}{(n)_2} \sum_{(i,j)\in \lbrac 1,n\rbrac}\ind_{i\sim_{G_n} j}- \int_{[0,1]^2} \widehat{\chi}_n(x_1,x_2) \ dx_1 dx_2\right|\\
	\begin{aligned}
	\leq & \left|\frac{1}{(n)_2} \sum_{(i,j)\in \lbrac 1,n\rbrac}\left(\ind_{i\sim_{G_n} j}-\widehat{\pi}_{Z_i,Z_j}\right) \right|\\
	& +\left|\frac{1}{(n)_2} \sum_{(i,j)\in \lbrac 1,n\rbrac}\widehat{\pi}_{Z_i,Z_j} -  (\widehat{\lambda}^n_1)^2 \widehat{\pi}_{11}^n- 2 \widehat{\lambda}^n_1(1-\widehat{\lambda}^n_1) \widehat{\pi}_{12}^n - (1-\widehat{\lambda}^n_1)^2 \widehat{\pi}_{22}^n\right|
	\end{aligned}
	\end{multline*}
	\begin{align*}
	\leq & \left|\frac{1}{(n)_2} \sum_{(i,j)\in \lbrac 1,n\rbrac}\left(\ind_{i\sim_{G_n} j}-\widehat{\pi}_{Z_i,Z_j}\right) \right|
	+\left|\widehat{\pi}^n_{11}\left(\sum_{(i,j)\ |\ (Z_i,Z_j)=(1,1)} \frac{1}{(n)_2}- (\widehat{\lambda}^n_{1})^2\right)\right|\\
	& +\left|\widehat{\pi}^n_{22}\left(\sum_{(i,j)\ |\ (Z_i,Z_j)=(2,2)}\frac{1}{(n)_2}- (1-\widehat{\lambda}^n_{1})^2\right)\right|\\
	&+\left|\widehat{\pi}^n_{12}\left(\sum_{\substack{(i,j)\ |\ (Z_i,Z_j)=(1,2) \\ \mbox{or} (Z_i,Z_j)=(2,1)}}\frac{1}{(n)_2}- 2\widehat{\lambda}^n_{1}(1-\widehat{\lambda}^n_{1})\right)\right|.
	\end{align*}
	By the law of large numbers and using \eqref{eq:cv_rho_gamma} whose proof does not depend on the Proposition \ref{prop:GnChin}, the four terms converge to zero.\\
	
	In the general case, proceeding in a similar way leads to:
	\begin{align*}
	|t(F, G_n) - t(F, \chi_n)| 	\leq &   \left| \frac{1}{(n)_k}\sum_{(i_1,\cdots i_k)\in \lbrac 1,n\rbrac}\prod_{\{\ell,\ell'\}\in E(F)}\ind_{i_\ell \sim_G i_{\ell'}} \right.\\
	&\left. - \frac{1}{(n)_k} \sum_{(i_1,\cdots ,i_k)} \prod_{\{\ell,\ell'\}\in E(F)} \left( \sum_{q,r =1}^Q \widehat{\pi}_{qr}^n \ind_{Z_{i_\ell}=q,Z_{i_{\ell'}}=r} \right) \right|\\
	& 	+ \left| \frac{1}{(n)_k} \sum_{(i_1,\cdots ,i_k)} \prod_{\{\ell,\ell'\}\in E(F)} \left( \sum_{q,r =1}^Q \widehat{\pi}_{qr}^n \ind_{Z_{i_\ell}=q,Z_{i_{\ell'}}=r} \right) \right.\\
	&\left. - \frac{1}{n^k} \sum_{1\leq i_1,\cdots ,i_k \leq n} \prod_{\{\ell,\ell'\}\in E(F)} \left( \sum_{q,r =1}^Q \widehat{\pi}_{qr}^n \ind_{Z_{i_\ell}=q,Z_{i_{\ell'}}=r} \right) \right|\\
	& + \left| \frac{1}{n^k} \sum_{1\leq i_1,\cdots ,i_k \leq n} \prod_{\{\ell,\ell'\}\in E(F)} \left( \sum_{q,r =1}^Q \widehat{\pi}_{qr}^n \ind_{Z_{i_\ell}=q,Z_{i_{\ell'}}=r} \right)  \right.\\
	& \left.-  \int_{[0,1]^k} \prod_{\{\ell,\ell'\}\in E(F)} \left(\sum_{q,r =1}^Q\widehat{\pi}_{qr}^n \ind_{J^n_q\times J^n_r}(x_\ell,x_{\ell'}) \right) dx_1\cdots dx_k   \right|
	\end{align*}

	As $\prod_{\{\ell,\ell'\}\in E(F)}\ind_{i_\ell \sim_G i_{\ell'}} $ and $\prod_{\{\ell,\ell'\}\in E(F)} \left( \sum_{q,r =1}^Q \widehat{\pi}_{qr}^n \ind_{Z_{i_\ell}=q,Z_{i_{\ell'}}=r} \right)$ are bounded by $1$, there exist $c(k)$ such that the first term and the second term in the right hand side are bounded by $c(k)/n$.
	For the third term, it is equal to
	\begin{align*}
	\left| \sum_{1 \leq q_1,...,q_k\leq Q}  \prod_{\{\ell, \ell '\} \in E(F)} \widehat{\pi}_{q_\ell,q_{\ell'}}^n \left(  \frac{1}{n^k} \sum_{1\leq i_1,\cdots ,i_k \leq n} \ind_{Z_{i_1}=q_{i_1},\cdots, Z_{i_k} = q_{i_k}} \hspace*{2cm}\right.
	\right.\\ \left.  \left. - \int_{[0,1]^k} \prod_{h=1}^{k} \ind_{J_{q_h}^n} (x_h) dx_1 \cdots dx_k \right) \right|	
	\end{align*}
	Since $0 \leq \prod_{\{\ell, \ell '\} \in E(F)} \widehat{\pi}_{q_\ell,q_{\ell'}}^n\leq 1$ and  $\{ Z_{i_1}=q_{i_1},\cdots, Z_{i_k} = q_{i_k} \} = \{ \Gamma(X_{i_1}) \in J_{q_1}, \cdots, \Gamma(X_{i_k}) \in J_{q_k} \}$, the third term is thus bounded by
	\begin{multline*}
	\sum_{1 \leq q_1,...,q_k\leq Q}  \left| \frac{1}{n^k} \sum_{1\leq i_1,\cdots ,i_k \leq n} \ind_{\Gamma(X_{i_1}) \in J_{q_1}, \cdots, \Gamma(X_{i_k}) \in J_{q_k} } -  \int_{[0,1]^k} \prod_{h=1}^{k} \ind_{J_{q_h}^n} (x_h) dx_1 \cdots dx_k \right|\\
	\begin{aligned}
	= &  \sum_{1 \leq q_1,...,q_k\leq Q} \left| \frac{1}{n^k}\sum_{1\leq i_1,\cdots ,i_k \leq n} \prod_{\ell=1}^k \ind_{\Gamma(X_{i_\ell})\in J_{i_\ell}} - \prod_{\ell=1}^k \int_{[0,1]} \ind_{J^n_{i_\ell}}dx_\ell  \right|\\
	= & \sum_{1 \leq q_1,...,q_k\leq Q} \left| \frac{\prod_{\ell=1}^{k} \sum_{i_\ell=1}^n \ind_{\Gamma(X_{i_\ell}) \in J_{q_l}}}{n^k} - \prod_{\ell=1}^k \int_{J^n_{q_\ell}}dx_\ell \right|\\
	= & \sum_{1 \leq q_1,...,q_k\leq Q} \left| \prod_{\ell=1}^k \frac{N_n^{q_\ell}}{n} - \prod_{\ell=1}^k \widehat{\lambda}_{q_\ell}^n\right| =0.
	\end{aligned}
	\end{multline*}
	Hence $\lim_{n\rightarrow +\infty} |t(F,G_n)-t(F,\chi_n)|=0$. Because $t(F,G_n)$ and $t(F,\chi_n)$ are bounded independently from $n$, this provides the announced result.

	\subsection{Incomplete observations and graphon de-biasing}\label{sec:incomplete_graphon}
	
\subsubsection{Case where $Z_i$ is unobserved but $X_i$ is}\label{sec:incompletegraphon1}
	
	In Proposition \ref{eq:estimatorsAR}, it is shown that the `classical' SBM estimator \eqref{estimateurs:classiques} obtained by neglecting the bias coming from the sampling scheme can be corrected by using the inverse of the cumulative distribution function $\Gamma$ of $m$. When the types are unobserved, we proceed in the same way. We assume here that the types $Z_i$ are unobserved, but we need the observation of the marks $X_i$, otherwise no de-biasing is permitted since the cumulative distribution function $\Gamma$ can not be estimated. We detail this estimation procedure in the case $Q=2$ for the sake of simplicity, but generalization is straightforward.\\
	
	\noindent \textbf{Step 1:} First, we perform an estimation of the SBM neglecting the sampling biases.
	\begin{itemize}
		\item We follow the algorithm described in Section \ref{sec:SAEM}, but with the likelihood $\mathcal{L}^{\class}(Z,Y ; \theta)$ given in \eqref{vraisemblance-classique}. We denote the parameter here by $\theta=(\lambda_1,1-\lambda_1 , \pi_{11}, \pi_{12},\pi_{21},\pi_{22})$.\\
		\item For the proposal distribution of the types $Z^c$, it is simpler since we assume that the $X_i$'s are known. Assume that we are at step $k$ and that we dispose of the parameters $\theta^{(k-1)}$. We initialize the types by attributing the types 1 to the $X_i \leq \lambda^{(0)}$ and 2 to the others. At each step, the threshold is modified from $\lambda^{(k-1)}_1$ to $\lambda^{(k)}_1$ by following a random walk: a gaussian increment (mean 0 and variance $s^2$) is added. All the $X_i$ smaller than this increment are given the type $Z_i=1$ and the others the type $Z_i=2$.
	\end{itemize}
	
	This Step 1 corresponds to a variational EM for the classical likelihood \eqref{vraisemblance-classique}, for which the consistency and asymptotic normality have been established by C\'elisse et al. \cite{celissdaudinpierre2012} and Bickel et al. \cite{bickelchoichangzhang}.	\\
	
	\noindent \textbf{Step 2:} We estimate the cumulative distribution function $\Gamma_n$ (see \eqref{eq:Gamman}) and deduce the graphon estimator $\widehat{\alpha}^n_1$ of $\alpha_1$ using \eqref{eq:widehatalpha}. This provides the estimator of $\kappa_{\theta}$:
	\begin{equation}
	\widehat{\kappa}_n(x,y):= \sum_{q=1}^Q \sum_{r=1}^Q \widehat{\pi}^n_{qr} \ind_{[\sum_{k=1}^{q-1} \widehat{\alpha}^n_k,\sum_{k=1}^{q} \widehat{\alpha}^n_k)}(x)\ind_{[\sum_{k=1}^{r-1} \widehat{\alpha}^n_k,\sum_{k=1}^{r} \widehat{\alpha}^n_k)}(y).
	\end{equation}

	\subsubsection{Case where both $X_i$ and $Z_i$ are unobserved}\label{sec:Referee}

When both $X_i$ and $Z_i$ are unobserved, it is not possible to compute the empirical cumulative distribution function $\Gamma_n$ any more. Thus, Equation \eqref{eq:widehatalpha} can not be used any more to obtain an estimator of $\alpha_q$ from an estimator of $\widetilde{\alpha}_q$.

As pointed out by an anonymous Referee, from \eqref{notation:pibar} and \eqref{def:alphatilde}, we can write that
\begin{equation}\label{eq:referee}
\widetilde{\alpha}_q=\frac{\alpha_q \bar{\pi}_q}{\bar{\pi}} ,\ \mbox{ for all }q\in \{1,\dots Q\}\quad \Leftrightarrow   \tilde{\alpha} = \frac{\alpha \odot (\pi \alpha)}{\alpha^{T} \pi \alpha},
\end{equation}in vectorial form, where $\odot$ is the Kronecker product of two vectors.
Then an estimator $\widehat{\alpha}$ for the vector $\alpha=(\alpha_1,\dots \alpha_Q)$ can be obtained from solving the equation:
\begin{align}\label{eq:alphaetpi}
\big(\widehat{\alpha}^{T} \widehat{\pi} \widehat{\alpha}\big) \widehat{\lambda} =\widehat{ \alpha} \odot (\widehat{\pi} \widehat{\alpha}).
\end{align}

\paragraph{For $Q=2$:} In this case, under the constraint $\widehat{ \alpha}_1 + \widehat{ \alpha}_2 = 1$, and equation \eqref{eq:alphaetpi} is written simply as:
\begin{align*}
\widehat{\lambda}_1 = \frac{\widehat{ \alpha}_1\big( \widehat{\pi}_{11} \widehat{\alpha}_1 + \widehat{\pi}_{12}(1-\widehat{\alpha}_1)  \big)}{\widehat{\pi}_{11}\widehat{\alpha}_1^2 + 2\widehat{\pi}_{12}\widehat{\alpha}_1(1- \widehat{\alpha}_1) + \widehat{\pi}_{22}(1- \widehat{\alpha}_1)^2}.
\end{align*}
It leads to a quadratic equation of $\widehat{\alpha}_1$ as follow:
\begin{align*}
	\bigg[  \big(  \widehat{\pi}_{11} + \widehat{\pi}_{22} - 2\widehat{\pi}_{12} \big) \widehat{\lambda} - (\widehat{\pi}_{11}- \widehat{\pi}_{12}) \bigg] \widehat{\alpha}_1^2 + \bigg[ 2( \widehat{\pi}_{12} -  \widehat{\pi}_{22}) \widehat{\lambda} -  \widehat{\pi}_{12} \bigg] \widehat{\alpha}_1 +  \widehat{\pi}_{22} \widehat{\lambda} =0.
\end{align*}Solving this second order equation,
\begin{equation}
\Delta
=\pi_{12}^2(2\lambda-1)^2+4\pi_{11}\pi_{22}\lambda(1-\lambda)\geq 0.
\end{equation}Hence, there are two solutions:
\begin{align}
\widehat{\alpha}_1=-\frac{\bigg[2( \widehat{\pi}_{12} -  \widehat{\pi}_{22}) \widehat{\lambda} -  \widehat{\pi}_{12} \bigg]\pm \sqrt{\pi_{12}^2(2\lambda-1)^2+4\pi_{11}\pi_{22}\lambda(1-\lambda)}}{2\big(\widehat{\pi}_{11} + \widehat{\pi}_{22} - 2\widehat{\pi}_{12} \big) \widehat{\lambda} - (\widehat{\pi}_{11}- \widehat{\pi}_{12})\big)}.
\end{align}
These solutions can be computed numerically.\\

\paragraph{For $Q\geq 3$:} Equation \eqref{eq:alphaetpi} is written as:
 $
	\big(\widehat{\alpha}^{T} \widehat{\pi} \widehat{\alpha} \big) \widehat{\lambda} - \widehat{\alpha} \odot (\widehat{\pi}\widehat{\alpha}) = 0.
$
Consider the function $g$ defined on $S= \{x= (x_1, \cdots, x_Q)\in [0;1]^Q: x_1+ ...+x_Q = 1\}$, \[g(x)= \big(\x^{T}\widehat{\pi} x \big) \widehat{  \lambda}- x \odot(\widehat{\pi} x). \]
It leads to solve the optimization problem
\[  \min_{x \in S} \|g(x)\|. \]

	\section{Numerical results}\label{sec:numresults}

	For the simulation, we consider RDS graphs obtained from the exploration of SBM graphons with $Q= 2$ classes, of respective proportions $\alpha_1=2/3$ and $\alpha_2=1/3$. The connection probabilities are:
	\[ \pi=\left(\begin{array}{cc}
	0.7 & 0.4 \\
	0.4 & 0.8\end{array}\right).\]
	The RDS graphs consist of $n = 50$ vertices. \\

	We proceed to the four estimations presented in this paper: 	
	\begin{itemize}
		\item Maximum likehood on complete data: the algorithm of Section \ref{sec:completeobs} for complete observations by assuming that the types $Z_i\in \{1 , 2\}$ are observed.
		\item SAEM: the algorithm of Section \ref{sec:SAEM} when the types $Z_i$ are unobserved. The SAEM is based on an iteration on $k$ and we perform $K=200$ iterations.
		\item De-biased graphon: the computation of the estimators given in Proposition \ref{eq:estimatorsAR} assuming complete observations,
		\item De-biased graphon with SAEM: again, we use an SAEM algorithm for the likelihood \eqref{vraisemblance-classique}, and then use the same de-biaising technique as in item 3 above (see Section \ref{sec:incomplete_graphon}). Again, we use $K=200$ iterations for the SAEM iterations.
\item De-biaised graphon by solving the algebraic equation for $\alpha_q$ \eqref{def:alphatilde}.
	\end{itemize}
	We proceed to a Monte-Carlo study of the estimators' distributions. We simulate 200 RDS graphs, and for each of them, apply the four estimation strategies. The empirical distribution of the estimators are represented in Fig. \ref{fig:complete}, and this allows us to estimate the associated mean squares errors (MSE) for each method, see Table \ref{tab:MSE-complete-obs_CI_exactes}.

	\begin{figure}[!ht]
		\centering
		\begin{tabular}{cc}
			\includegraphics[width=6cm,height=6.5cm]{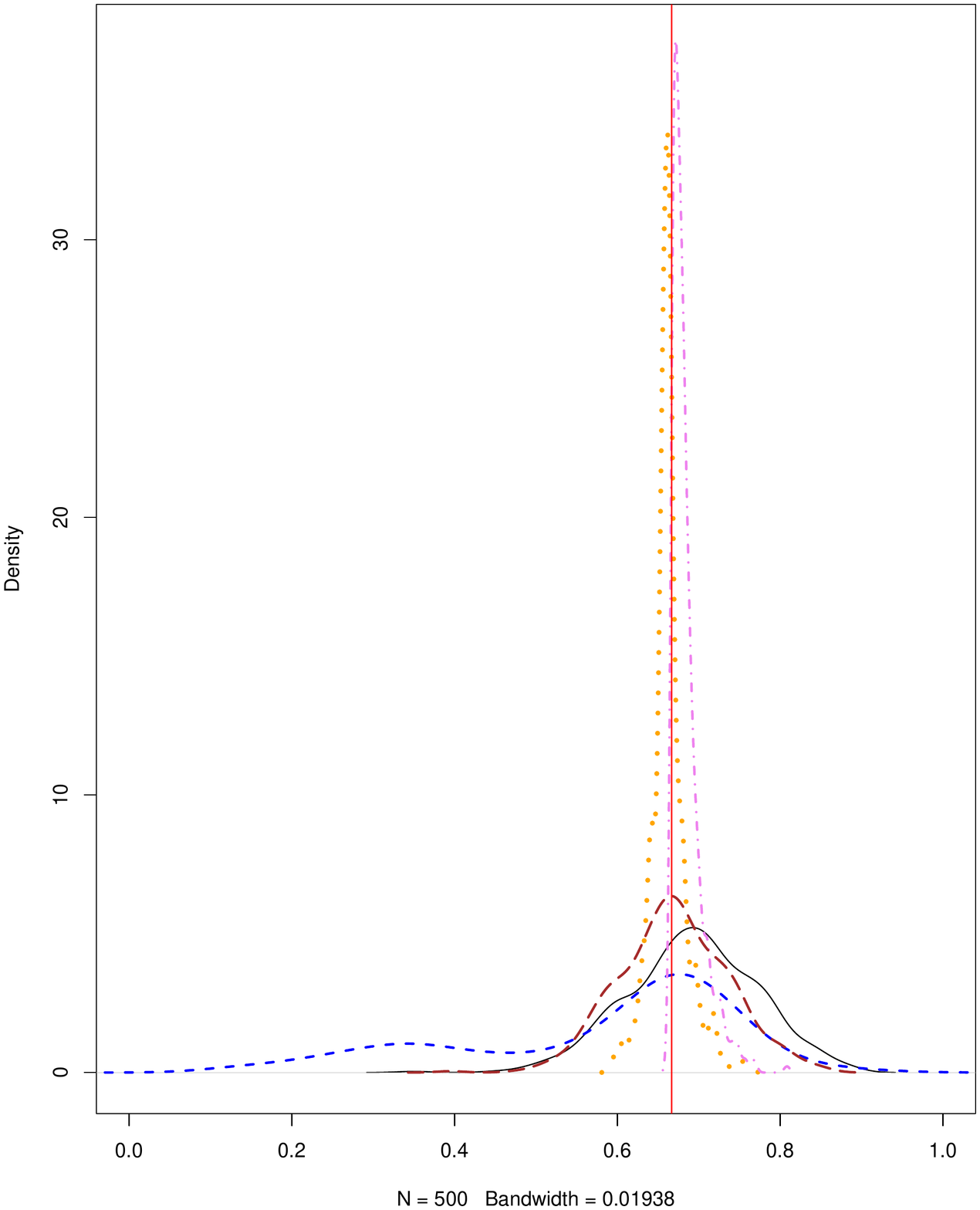} & \includegraphics[width=6cm,height=6.5cm]{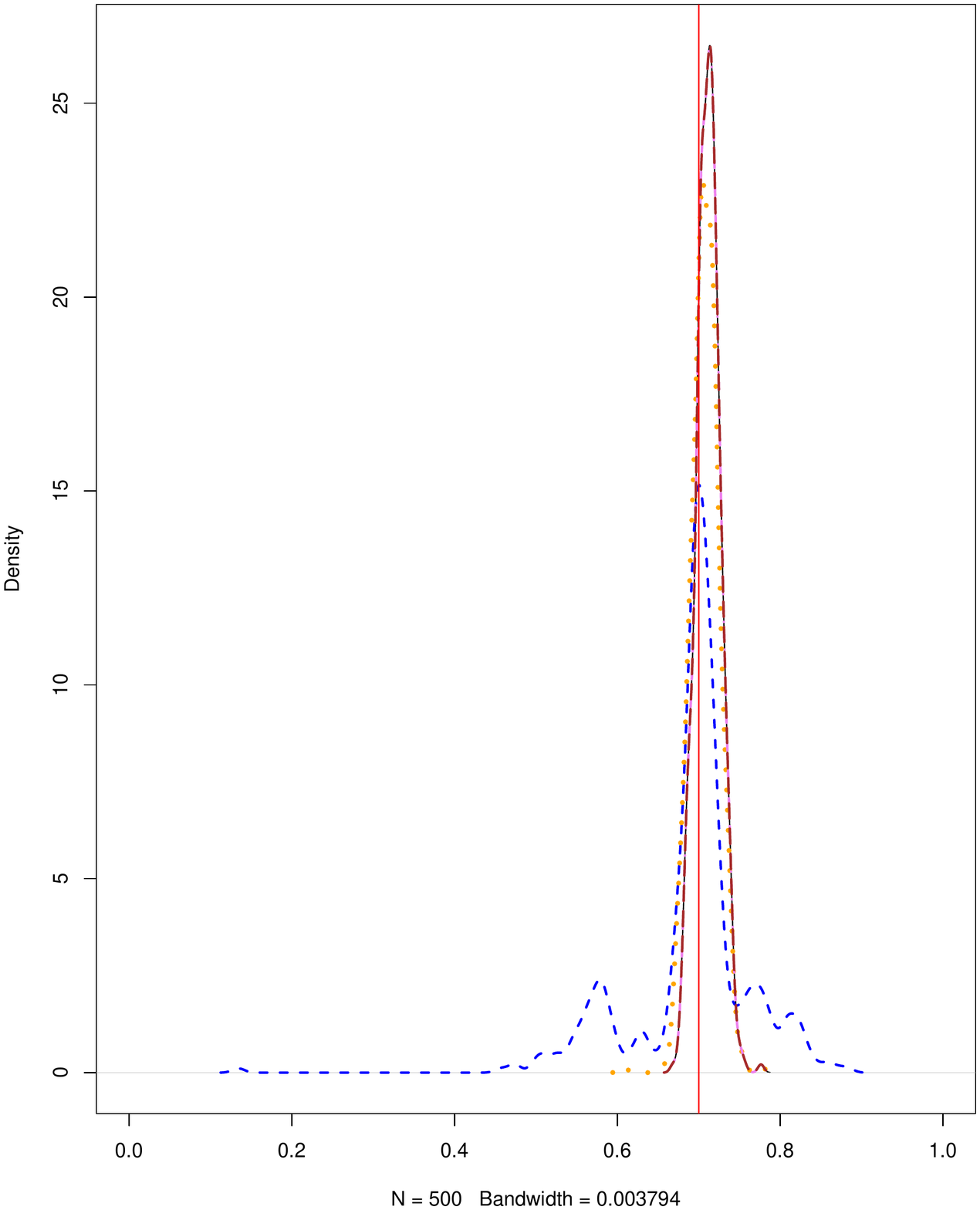}\\
			(a): $\alpha$ & (b): $\pi_{11}$ \\
			\includegraphics[width=6cm,height=6.5cm]{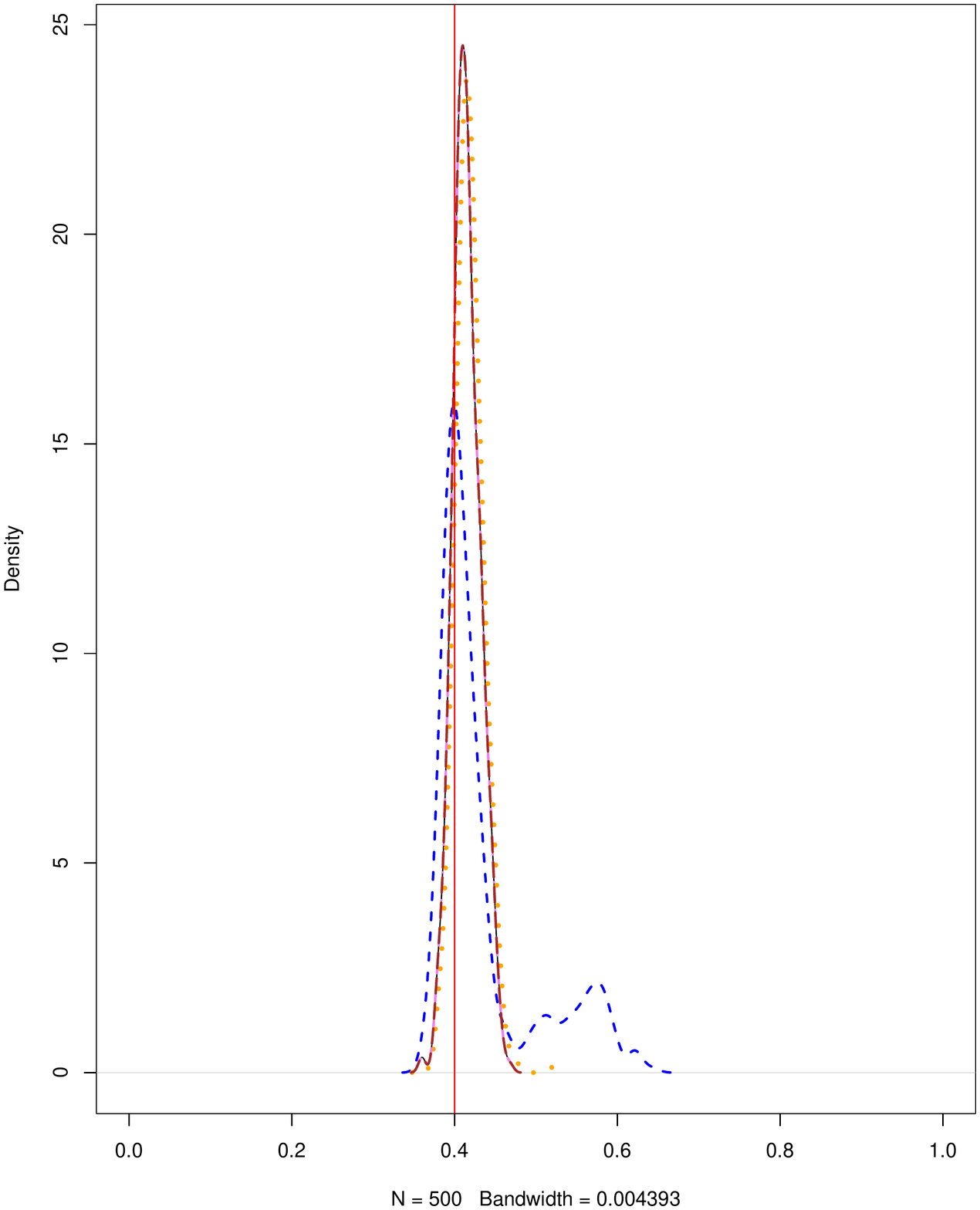} & \includegraphics[width=6cm,height=6.5cm]{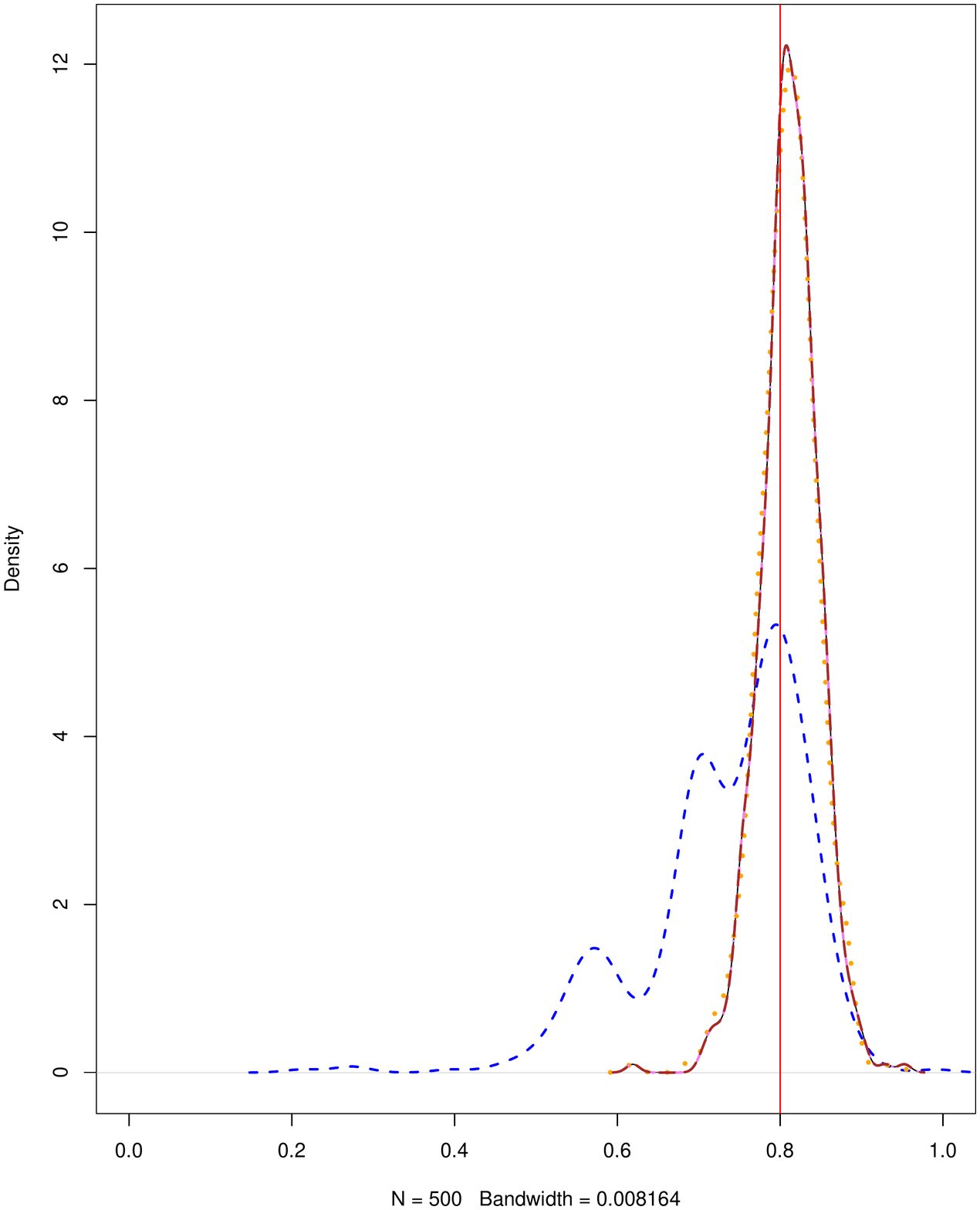}\\
			(c): $\pi_{12}$ & (d): $\pi_{22}$
		\end{tabular}
		\caption{{\textit \tvc{Estimation on complete data for a graph of $n=60$ vertices with $Q=2$ classes and parameters $\alpha_1=2/3$, $\pi_{11}=0.7$, $\pi_{12}=\pi_{21}=0.4$ and $\pi_{22}=0.8$. 500 such graphs are simulated and the empirical distributions of the estimators are represented here with the true parameters in red line. On each graph: the MLE with complete observation (Section \ref{sec:completeobs}) is in continuous black line, the SAEM estimator (Section \ref{sec:likelihood-SAEM}) is in blue dashed line, the graphon estimator with complete observation (Section \ref{sec:completeobs2}) is in dash-dotted pink line, the graphon estimator with incomplete observation and SAEM algorithm (Section \ref{sec:incompletegraphon1}) is in yellow dotted line, the  graphon estimator with incomplete observation and algebraic equations (Section \ref{sec:Referee}) is in brown long-dashed line. The Graphon (a): estimator of $\alpha$, (b): estimator of $\pi_11$, (c): estimator of $\pi_{12}$, (d) estimator of $\pi_{22}$.} }}\label{fig:complete}
	\end{figure}

\tvc{	\begin{table}[!ht]
		\centering
		\begin{tabular}{c|cc|ccc}
			& Complete & SAEM & De-biased & De-biased & De-biased\\
			Parameters & likelihood  & & graphon & graphon with SAEM & graphon with alg. eq. \\
			\hline
			$\pi_{11}$ & $3.52\ 10^{-4}$  & $5.25\ 10^{-3}$ & $3.52\ 10^{-4}$  & $3.54\ 10^{-4}$ & $3.54\ 10^{-4}$ \\
			$\pi_{12}$ & $4.99\ 10^{-4}$  & $5.14\ 10^{-3}$  & $4.99\ 10^{-4}$  & $6.65\ 10^{-4}$ & $4.99\ 10^{-4}$\\
			$\pi_{22}$ & $1.41\ 10^{-3}$ &  $1.45\ 10^{-2}$  & $1.41\ 10^{-3}$  & $1.42\ 10^{-3}$& $1.41\ 10^{-3}$\\
			$\alpha$ &   $7.01\ 10^{-3}$ &  $3.80\ 10^{-2}$  & $6.80\ 10^{-4}$  & $5.31\ 10^{-4}$ & $4.51\ 10^{-3}$
		\end{tabular}
		\caption{\textit{Mean square errors. }}\label{tab:MSE-complete-obs_CI_exactes}
	\end{table}
	}
\tvc{	Without surprise, for the maximum likelihood estimation, the estimation is better when we have complete observations (compare columns 1 and 2). Note that the use of the SAEM algorithm could be accelerated, which is discussed in the conclusion. For the graphon de-biasing, the methods with incomplete observations perform well, sometimes equally to the methods with complete observations.\\
	}
	
	When the types $Z_i$ are not observed, we achieve better MSEs with the debiasing of the classical SAEM method of Daudin et al. (column 4 of Table \ref{tab:MSE-complete-obs_CI_exactes}). Notice first that the columns 2 and 4 of Table \ref{tab:MSE-complete-obs_CI_exactes} are not completely equivalent, since the debiasing methods of Section \ref{sec:graphon} necessitate the knowledge of the positions $X_i$ of the Markov chain, when the likelihood \eqref{eq:likelihood-in-fctofN} necessitates only the connections $Y_{ij}$ and the types $Z_i$'s. Second, the updating of the types in the SAEM algorithm is easier in Section \ref{sec:incomplete_graphon} when the $X_i$'s are known since it amounts to choosing the threshold that separates the types 1 and 2. Finally, the SAEM algorithm on the classical likelihood \eqref{vraisemblance-classique} seems to converge more easily than for the likelihood \eqref{eq:likelihood-in-fctofN}.
	
	\section{Conclusion}
	
	\tvc{	Four statistical methods are studied in this paper, for estimating SBM parameters using a subgraph obtained from the exploration of the graphon by a Markov chain:
		\begin{itemize}
			\item Two methods built on the maximum likelihood.
			\begin{itemize}
				\item The first one is the classical maximum likelihood estimator on the complete data, and necessitates the observation of the types $Z_i$'s and the edges of $G_n$, $Y_{ij}$'s. See Section \ref{sec:completeobs}.
				\item The second method is an SAEM estimation procedure that can be used when only the connectivities $Y_{ij}$'s are observed.
			\end{itemize}
			\item Three methods built on the de-biasing formula of Athreya and R\"ollin \cite{athreyaroellin1}.
			\begin{itemize}
				\item The first one is on the complete data, and necessitates the observation of the positions $X_i$'s, the types $Z_i$'s and the edges of $G_n$, $Y_{ij}$'s. See Section \ref{sec:completeobs2}.
				\item The second method is a variation started from the SAEM estimation procedure of Daudin et al. when there is no sampling bias. The latter estimation can be used when only the connectivities $Y_{ij}$'s are observed, but the de-biasing using the cumulative distribution function $\Gamma$ needs information on the positions $X_i$'s (but not the complete knowledge of the types $Z_i$'s).
\item The last one solves an algebraic equation satisfied by the $\alpha_q$'s and obtained from \eqref{def:alphatilde}. This method does not require the knowledge of the $X_i$'s but only of the $Y_{ij}$'s.
			\end{itemize}
		\end{itemize}
	}
	This is a toy model for estimating random networks from chain-referral sampling techniques and there exist sampling biases. The two first methods compute the maximum likelihood estimator when the types of the nodes are known or unknown. On simulations, it appears that the SAEM algorithm used when the types are unobserved is not very robust and provides relatively large MSEs. \tvc{However, the relatively rough SAEM algorithm that we use here might be improved by using Metropolis-Hastings and Gibbs algorithms with refined exploration of the state space of the $Z_i$'s. }\\
	An alternative approach is proposed by taking advantage of recent results by Athreya and R\"ollin \cite{athreyaroellin1}: this allows to correct the classical SBM estimators that would be proposed if one ignores the sampling biases. These methods provide good estimators but rely on the precise knowledge of the Markov chain exploring the SBM graphon (in particular the positions $X_i$'s), which is not always available.

\end{document}